\documentclass[11pt,reqno]{amsart}

\usepackage{amsmath, amsfonts, amssymb,  mathrsfs,  array, stmaryrd,  indentfirst, amsthm,  hyperref, comment}
\usepackage{graphicx, enumitem, tabularx, color}

\usepackage[margin=1.0in]{geometry}
\usepackage{setspace}
\numberwithin{equation}{section}
\theoremstyle{plain}
\newtheorem{lemma}{Lemma}[section]
\newtheorem{thm}{Theorem}[section]
\newtheorem{defn}{Definition}[section]
\newtheorem{prop}{Proposition}[section]

\newtheorem{remark}{Remark}[section]
\newcommand{\lbar}[1]{\underline{#1}}
\newcommand{\ubar}[1]{\overline{#1}}

\newcommand{\bfDelta}{\mathbf{\Delta}}

\makeatletter
\def\@setcopyright{}
\def\serieslogo@{}
\makeatother

\begin{document}

\title[]{Stochastic Perron's Method for the Probability of lifetime ruin problem under transaction costs}\thanks{This research is supported by the National Science Foundation under grant  DMS-0955463.}

\author{Erhan Bayraktar}
\address[Erhan Bayraktar]{Department of Mathematics, University of Michigan, 530 Church Street, Ann Arbor, MI 48109, USA}
\email{erhan@umich.edu}
\author{Yuchong Zhang }
\address[Yuchong Zhang]{Department of Mathematics, University of Michigan, 530 Church Street, Ann Arbor, MI 48109, USA}
\email{yuchong@umich.edu}

\begin{abstract}
We apply stochastic Perron's method to a singular control problem where an individual targets at a given consumption rate, invests in a risky financial market in which trading is subject to proportional transaction costs, and seeks to minimize her probability of lifetime ruin. Without relying on the dynamic programming principle (DPP), we characterize the value function as the unique viscosity solution of an associated Hamilton-Jacobi-Bellman (HJB) variational inequality. We also provide a complete proof of the comparison principle which is the main assumption of stochastic Perron's method.
\end{abstract}

\maketitle
\pagestyle{headings}

\section{Introduction}
Stochastic Perron's method is introduced in \cite{BayraktarSirbu12}, \cite{BayraktarSirbu14} and \cite{BayraktarSirbu13} as a way to show the value function of a stochastic control problem is the unique viscosity solution of the associated Hamilton-Jacobi-Bellman (HJB) equation, without having to first go through the proof of the dynamic programming principle (DPP) which is usually very long and complicated, and often incomplete. It is a direct verification approach in that it first constructs a solution to the HJB equation, and then verifies such a solution is the value function. But unlike the classical verification, it does not require regularity; uniqueness acts as a substitute for verification.
The basic idea is to define, for each specific problem, a suitable family of stochastic supersolutions $\mathcal{V}^+$ (resp. stochastic subsolutions $\mathcal{V}^-$) which is stable under minimum (resp. maximum), and whose members bound the value function from above (resp. below). So the value function is enveloped from above by $v_+=\inf_{v\in\mathcal{V}^+} v$ and from below by $v_-=\sup_{v\in\mathcal{V}^-} v$. The key step is to show $v_+$ is a viscosity subsolution and $v_-$ is a viscosity supersolution by a Perron-type argument. A comparison principle then closes the gap. 

Stochastic Perron's method has been applied to linear problems \cite{BayraktarSirbu12}, Dynkin games \cite{BayraktarSirbu14}, HJB equations for regular control problems \cite{BayraktarSirbu13}, (regular) exit time problems \cite{Rokhlin14a} and zero-sum differential games \cite{Sirbu14}. This paper adapts the method to another type of problems: singular control problems. In particular, we focus on the specific problem of how individuals should invest their wealth in a risky financial market to minimize the probability of lifetime ruin, when buying and selling of the risky asset incur proportional transaction costs. This problem can also be treated as an exit time problem, but with singular controls. In the frictionless case, the probability of lifetime ruin problem was analyzed by Young \cite{Young04}, and later studied in more complicated settings such as borrowing constraints \cite{BayraktarYoung07b}, stochastic consumption \cite{BayraktarYoung11} and drift uncertainty \cite{Yuchong14}. So the goal of the paper is two-fold. First, it exemplifies how stochastic Perron's method can be applied to singular control problems, which has not been covered in the literature. Second, it serves as the first step towards a rigorous analysis of the probability of lifetime ruin problem under transaction costs. The techniques in this paper can be applied in a similar way to other optimal investment problems under transaction costs, as long as there is a comparison principle. For consumption-investment problems, uniqueness is proved in \cite{Kabanov04} under certain conditions (also see \cite[Theorem 1]{TZ95} and Section 4.3 of \cite{Kabanov}).

The main idea of the proof is in line with \cite{BayraktarSirbu13} and \cite{Rokhlin14a}, but there are some nontrivial modifications. Similar to \cite{DavisNorman90} and \cite{ShreveSoner94}, our HJB equation takes the form of a variational inequality with three components, one for each of the three different regions: no-transaction, sell, and buy. This makes the proof of the interior viscosity subsolution property of the upper stochastic envelope $v_+$ more demanding: we have to argue by contradiction in three cases separately. Variational inequalities also appear in \cite{BayraktarSirbu14} and the authors are able to rule out some of the cases by assuming the existence of a stochastic supersolution (resp. subsolution) less than or equal to the upper obstacle (resp. greater than or equal to the lower obstacle). But the same idea does not work for gradient constraints. Another challenge posed by the singular control is that the state process can jump outside the small neighborhood in which local estimates obtained from the viscosity solution property are valid. This issue arises in the proof of the interior viscosity supersolution property of the lower stochastic envelope $v_-$, and we overcome it by splitting the jump into two steps: first to an intermediate point on the boundary of the neighborhood and then to its original destination. 

In proving the viscosity semi-solution property of $v_{\pm}$, boundary property is usually harder to show than interior property. In fact, most of the work in \cite{Rokhlin14a} is devoted to proving the boundary viscosity semi-solution property of $v_{\pm}$. In our case, we avoid this hassle by constructing explicitly a stochastic supersolution and a stochastic subsolution
both of which satisfy the boundary condition. The boundary viscosity semi-solution property then becomes a trivial consequence of the definition of $v_\pm$. This is very similar to classical Perron's method in which one has to first come up with a pair of viscosity semi-solutions satisfying the boundary condition (see Theorem 4.1 and Example 4.6 of \cite{UsersGuide}). However, we point out that the construction of such stochastic semi-solutions depends on the specific problem at hand and may not always be possible.

Previous works on stochastic Perron's method focus on methodology and take comparison principle (which is crucial for stochastic Perron's method to work) as an assumption. Here we provide, in addition to stochastic Perron's method, a complete proof of the comparison principle for our specific singular control problem. The proof relies on the existence of a strict classical subsolution satisfying certain growth condition, an idea we borrowed from \cite{Kabanov04}.

The rest of the paper is organized as follows. In Section~\ref{sec:formulation}, we set up the problem, derive the HJB equation and some bounds on the value function, and state the main theorem. In Section~\ref{sec:ssupsoln}, we introduce the notion of stochastic supersolution and show the infimum of stochastic supersolutions is a viscosity subsolution. In Section~\ref{sec:ssubsoln}, we introduce the notion of stochastic subsolution and show the supremum of stochastic subsolutions is a viscosity supersolution. Finally, in Section~\ref{sec:comparison} we prove a comparison principle
and finish the proof of the main theorem.

\section{Problem formulation}\label{sec:formulation}
Let $(\Omega, \mathcal{F},\mathbb{P})$ be a probability space supporting a Brownian motion $W=(W_t)_{t\geq 0}$ and an independent Poisson process $N=(N_t)_{t\geq 0}$ with rate $\beta$. Let $\tau_d$ be the first time that the Poisson process jumps, modeling the death time of the individual. $\tau_d$ is exponentially distributed with rate $\beta$, known as the hazard rate in this context. Denote by $\mathbb{F}:=\{\mathcal{F}_t\}_{t\geq 0}$ the completion of the natural filtration of the Brownian motion and $\mathbb{G}:=\{\mathcal{G}_t\}_{t\geq 0}$ the completion of the filtration generated by $W$ and the process $1_{\{t\geq\tau_d\}}$. Assume both $\mathcal{F}$ and $\mathcal{G}$ have been made right continuous; that is, they satisfy the usual condition.

The financial market consists of a risk-free money market with interest rate $r>0$ and a risky asset (a stock) whose price $P_t$ follows a geometric Brownian motion with drift $\alpha>r$ and volatility $\sigma>0$. Transferring assets between the money market and the stock market incur proportional transaction costs specified by two parameters $\lambda, \mu\in(0,1)$. One can think of the stock as having ask price $P_t/(1-\lambda)$ and bid price $(1-\mu)P_t$.
Same as \cite{ShreveSoner94}, we describe the investment policy of the individual by a pair $(B,S)$ of right-continuous with left limits (RCLL), non-negative, non-decreasing and $\mathbb{G}$-adapted processes, where $B$ records the cumulative amount of money withdrawn from the money market for the purpose of buying stock, and $S$ records the cumulative sales of stock for the purpose of investment in the money market. We set $(B_{0-},S_{0-})=\mathbf{0}$, i.e. there is no investment history at time zero. Due to transaction costs, it is never optimal to buy and sell at the same time. So we limit ourselves to strategies $(B,S)$ such that for all $t$, $\triangle B_t:=B_t-B_{t-}$ and $\triangle S_t:=S_t-S_{t-}$ are not both strictly positive. Denote by $\mathscr{A}_0$ the set of all such pairs $(B,S)$.
Apart from investment, the individual also consumes at a constant rate $c>0$. 

Denote by $X_t$ and $Y_t$ the total dollar amount invested in the money market and the stock at time $t$, respectively. Let $L(x,y):=x+(1-\mu)y^+- \frac{1}{1-\lambda}y^-$ be the liquidation function. For each $a\in\mathbb{R}$, define
\[\mathcal{S}_a:=\{(x,y)\in\mathbb{R}^2:L(x,y)>a\}=\{(x,y)\in\mathbb{R}^2: x+\frac{y}{1-\lambda}>a, x+(1-\mu)y>a\}.\]
Given initial endowment $(x,y)$ and a pair of control $(B,S)\in\mathscr{A}_0$, the pre-death investment position of the individual evolve according to the stochastic differential equations (SDE)
\begin{align}
&dX_t=(rX_{t}-c)dt-dB_t+(1-\mu)dS_t, \quad X_{0-}=x, \label{X}\\
&dY_t=\alpha Y_t dt+\sigma Y_t dW_t+(1-\lambda)dB_t-dS_t, \quad Y_{0-}=y. \label{Y}
\end{align}
Here we allow an immediate transaction at time zero so that $(X_0, Y_0)$ may differ from $(x,y)$. Denote the solution by $(X^{x,y,B,S}, Y^{x,y,B,S})$. Let
\[\tau_b^{x,y,B,S}:=\inf\{t\geq 0: (X^{x,y,B,S},Y^{x,y,B,S})\notin S_b\}\]
be the ruin time. The individual aims at minimizing the probability that ruin happens before death. The value function of this control problem is defined as
\begin{equation}
\psi(x,y):=\inf_{(B,S)\in\mathscr{A}_0}\mathbb{P}(\tau_b^{x,y,B,S}<\tau_d).
\end{equation}

Clearly, $\psi$ is $[0,1]$-valued, and $\psi(x,y)=1$ if $(x,y)\notin \mathcal{S}_b$. Same as in the frictionless case, when $L(x,y)\geq c/r$, the individual can sustain her consumption by immediately putting all her money in the money market and consuming the interest. We shall assume $b<c/r$, otherwise the problem is trivial.\footnote{If $b\geq c/r$, then $\psi(x,y)$ is either 0 or 1, depending on whether $(x,y)$ belongs to $\mathcal{S}_b$ or not.} We have $\psi(x,y)=0$ for $(x,y)\in \ubar{\mathcal{S}}_{c/r}$. In other words, $\ubar{\mathcal{S}}_{c/r}$ is a ``safe region". The (open) state space for this control problem is $\mathcal{S}:=\mathcal{S}_{b}\backslash \ubar{\mathcal{S}}_{c/r}$, and the boundary consists of two parts: the ruin level $\partial\mathcal{S}_b$ and the safe level $\partial\mathcal{S}_{c/r}$.

For $\varphi\in C^2(\mathcal{S})$, define
\[\mathcal{L}\varphi:=\beta \varphi-(rx-c)\varphi_x-\alpha y \varphi_y-\frac{1}{2}\sigma^2 y^2\varphi_{yy}.\]
The HJB equation for the frictional lifetime ruin problem is
\begin{equation}\label{HJB}
\max\left\{\mathcal{L}u, -(1-\mu)u_x+u_y, u_x-(1-\lambda)u_y\right\}=0, \quad (x,y)\in \mathcal{S},
\end{equation}
with boundary conditions
\begin{equation}\label{BC}
u(x,y)=1 \text{ if } (x,y)\in\partial \mathcal{S}_b, \quad u(x,y)=0 \text{ if } (x,y)\in\partial \mathcal{S}_{c/r}.
\end{equation}

\subsection{Upper and lower bounds on the value function}
Let
\begin{equation}\label{upperbdd}
\ubar{\psi}(x,y):=\left(\frac{c-rL(x,y)}{c-rb}\right)^{\frac{\beta}{r}}, \quad (x,y)\in\ubar{\mathcal{S}}.
\end{equation}
$\ubar{\psi}$ is the probability of ruin if the agent immediately liquidate her stock position and makes no further transaction throughout her lifetime. It is an upper bound for the value function since such a strategy may not be optimal. It is easy to see that $\ubar{\psi}$ satisfies the boundary conditions \eqref{BC}. 

For $k\in[1-\mu, \frac{1}{1-\lambda}]$, let
\begin{equation}\label{eq:psi_k}
\psi_k(x,y):= 
\begin{cases}
\left(\frac{c-r(x+ky)}{c-rb}\right)^{d} , & b\leq x+ky\leq c/r,\\
0, & x+ky> c/r.
\end{cases}
\end{equation}
where 
\begin{equation}\label{dR}
d=\frac{1}{2r}\left[(r+\beta+R)+\sqrt{(r+\beta+R)^2-4r\beta}\right]>1, \quad R=\frac{1}{2}\left(\frac{\alpha-r}{\sigma}\right)^2.
\end{equation}
That is, $\psi_k(x,y)$ is the minimum frictionless probability of ruin when the initial wealth is $x+ky$ (the frictionless ruin probability is derived in \cite{Young04}). $\psi_k$ bounds the frictional value function from below because each $k$ corresponds to a stock price inside the bid-ask spread, and trading at a more favorable frictionless price obviously leads to smaller ruin probability. For a rigorous proof, one can refer to Remark \ref{rmk:ssubsoln<=psi} and Lemma \ref{lemma:lbarpsi is a ssubsoln}. Since the value function $\psi$ is bounded from below by $\psi_k$ for each $k$, it is bounded from below by their supremum:
\begin{equation}\label{lowerbdd}
\lbar{\psi}(x,y):=\sup_{k\in[1-\mu,\frac{1}{1-\lambda}]} \psi_k(x,y)=\psi_{1-\mu}(x,y)\vee \psi_{\frac{1}{1-\lambda}}(x,y)=\left(\frac{c-rL(x,y)}{c-rb}\right)^{d}.
\end{equation}
Since $\psi_k$ is continuous in $k$, the above supremum remains unchanged if we replace $[1-\mu,\frac{1}{1-\lambda}]$ by $(1-\mu,\frac{1}{1-\lambda})\cap\mathbb{Q}$. Clearly, $\lbar{\psi}$ satisfies the boundary conditions \eqref{BC}. 

The following lemma summarizes the results.
\begin{lemma}
For $(x,y)\in\ubar{\mathcal{S}}$,
\[\left(\frac{c-rL(x,y)}{c-rb}\right)^{d}\leq\psi(x,y)\leq \left(\frac{c-rL(x,y)}{c-rb}\right)^{\frac{\beta}{r}},\]
where $d$ is defined in \eqref{dR}.
\end{lemma}
\begin{remark}
It can be shown that $\ubar{\psi}$ is a viscosity supersolution and $\lbar{\psi}$ is a viscosity subsolution of \eqref{HJB}. With a comparison principle which we will prove in Section~\ref{sec:comparison}, one can use (classical) Perron's method introduced by Ishii  \cite{Ishii87} (also described in \cite{UsersGuide}) to get the existence of a viscosity solution to \eqref{HJB}, \eqref{BC}. But such a solution cannot be compared with the value function unless one can prove regularity which is necessary for the classical verification theorem. Instead, we will use stochastic Perron's method which amounts to verification without smoothness.
\end{remark}

\subsection{Random initial condition and admissible controls}
For convenience in later discussion, we introduce a ``coffin state" $\bfDelta$. Let $\ubar{\mathcal{S}}\cup \bfDelta$ be the one point compactification of $\ubar{\mathcal{S}}$. Throughout this paper, all closures are taken in $\mathbb{R}^2$. For any $\mathbb{R}^2$-valued vector $z$, we use the convention that $\bfDelta+z=\bfDelta$. Set $(X_t, Y_t):=\bfDelta$ for all $t\geq \tau_d$. 
For any function $u$ defined on $\ubar{\mathcal{S}}$, define its extension to $\ubar{\mathcal{S}}\cup \{\bfDelta\}$ by assigning $u(\bfDelta)=0$. 

A pair $(\tau,\xi)$ is called a \textit{random initial condition} for \eqref{X}, \eqref{Y} if $\tau$ is a $\mathbb{G}$-stopping time taking values in $[0,\tau_d]$, $\xi=(\xi^0,\xi^1)$ is a $\mathcal{G}_{\tau}$-measurable random vector taking values in $\ubar{\mathcal{S}}\cup\{\bfDelta\}$, and $\xi=\bfDelta$ if and only if $\tau=\tau_d$. Denote by $(X^{\tau,\xi,B,S}, Y^{\tau,\xi,B,S})$ the solution of \eqref{X} and \eqref{Y} with random initial condition $(\tau,\xi)$ in the sense that $(X_{\tau-},Y_{\tau-})=\xi$. 
The \textit{exit time} of $(X^{\tau,\xi,B,S}, Y^{\tau,\xi,B,S})$ from $\mathcal{S}$ is defined by
\[\sigma^{\tau,\xi,B,S}:=\inf\{t\geq \tau: (X^{\tau,\xi,B,S}_t, Y^{\tau,\xi,B,S}_t)\notin \mathcal{S}\}.\]
Note that $\sigma^{\tau,\xi,B,S}\leq \tau_d<\infty$ since $(X^{\tau,\xi,B,S}_{\tau_d}, Y^{\tau,\xi,B,S}_{\tau_d})=\bfDelta\notin\mathcal{S}$. 

We also restrict ourselves to a subset of controls. Observe that when buying stocks, we move northwest along the vector $(-1,1-\lambda)$; when selling stocks, we move southeast along the vector $(1-\mu,-1)$. It is not hard to see by picture that
starting in $\mathcal{S}$, one can never jump to $\mathcal{S}_{c/r}$ by a transaction. On the other hand, it is never optimal to jump across $\partial\mathcal{S}_b$ from $\mathcal{S}$ because such a jump immediately leads to ruin. If we are on $\partial\mathcal{S}_{c/r}$ (resp. $\partial\mathcal{S}_{b}$), jumping to its right is impossible and jumping to its left is not optimal (resp. does not prevent ruin from happening). Therefore, we may focus on those controls under which the controlled process exits $\mathcal{S}$ via its boundary or the coffin state. The formal definition of admissibility is given below.

\begin{defn}\label{defn_adm}
Let $(\tau,\xi)$ be a random initial condition. A control pair $(B,S)\in\mathscr{A}_0$ is called $(\tau,\xi)-$\textit{admissible} if 
\[(X^{\tau,\xi,B,S}_{\sigma^{\tau,\xi,B,S}},Y^{\tau,\xi,B,S}_{\sigma^{\tau,\xi,B,S}})\in\partial\mathcal{S}\cup\{\bfDelta\}.\]
Denote the set of $(\tau,\xi)-$\textit{admissible} controls by $\mathscr{A}(\tau,\xi)$. 
\end{defn}
We have $(B,S)\equiv\mathbf{0}\in\mathscr{A}(\tau,\xi)$ for any random initial condition $(\tau,\xi)$. When $\tau=0$ and $\xi=(x,y)$, we shall omit the $\tau$-dependence in the superscripts of the controlled process and relevant stopping times, and write $\mathscr{A}(\tau,\xi)=\mathscr{A}(x,y)$. As we have argued, working with admissible controls does not change the optimal probability, i.e.
\begin{equation*}
\psi(x,y)=\inf_{(B,S)\in\mathscr{A}(x,y)}\mathbb{P}(\tau_b^{x,y,B,S}<\tau_d).
\end{equation*}

The following constructions of admissible controls will be used a few times in Section~\ref{sec:ssupsoln}. We list them here for future reference.
\begin{lemma} \ \label{lemma:admissibility}
\begin{itemize}
\item[(i)] If $(B^i, S^i)$, $i=1,2$ are $(\tau,\xi)$-admissible and $A$ is any $\mathcal{G}_\tau$-measurable set, then 
\[(B_t,S_t):=1_{\{t\geq\tau\}}\left[\left(B^1_t-B^1_{\tau-},S^1_t-S^1_{\tau-}\right)1_A+\left(B^2_t-B^2_{\tau-},S^2_t-S^2_{\tau-}\right)1_{A^c}\right]\]
is also $(\tau,\xi)$-admissible.
\item[(ii)] Let $(B^1, S^1)$ be a $(\tau,\xi)$-admissible control, $\tau_1\in[\tau,\sigma^{\tau,\xi,B^1,S^1}]$ be a $\mathbb{G}$-stopping time, and $\xi_1:=(X^{\tau,\xi,B^1,S^1}_{\tau_1}, Y^{\tau,\xi,B^1,S^1}_{\tau_1})$. Then $(\tau_1,\xi_1)$ is a random initial condition. Furthermore, let $(B^2, S^2)$ be a $(\tau_1,\xi_1)$-admissible control. Then
\[(B_t,S_t):=1_{\{t<\tau_1\}}(B^1_t,S^1_t)+1_{\{t\geq\tau_1\}}(B^2_t-B^2_{\tau_1-}+B^1_{\tau_1},S^2_t-S^2_{\tau_1-}+S^1_{\tau_1})\]
is a $(\tau,\xi)$-admissible control.
\end{itemize}
\end{lemma}
\begin{proof}
(i) $(B,S)$ is $\mathbb{G}$-adapted by the definition of stopping time and stopping time filtration, and the $\mathbb{G}$-adaptedness of $(B^i, S^i)$, $i=1,2$. It is nonnegative because $(B^i, S^i)$, $i=1,2$ are non-decreasing. Monotonicity, RCLL property and that $\triangle B$ and $\triangle S$ are not both strictly positive also follow from the assumption that $(B^i, S^i)\in\mathscr{A}_0$, $i=1,2$. So $(B,S)\in\mathscr{A}_0$. By pathwise uniqueness of the solution to \eqref{X}, \eqref{Y}, we have
\[(X^{\tau,\xi,B,S}_t,Y^{\tau,\xi,B,S}_t)=1_A (X^{\tau,\xi,B^1,S^1}_t,Y^{\tau,\xi,B^1,S^1}_t)+1_{A^c} (X^{\tau,\xi,B^2,S^2}_t,Y^{\tau,\xi,B^2,S^2}_t), \quad t\geq \tau.\]
It follows that
\[\sigma^{\tau,\xi,B,S}=1_A\sigma^{\tau,\xi,B^1,S^1}+1_{A^c}\sigma^{\tau,\xi,B^2,S^2},\]
and thus
\[(X^{\tau,\xi,B,S}_{\sigma^{\tau,\xi,B,S}},Y^{\tau,\xi,B,S}_{\sigma^{\tau,\xi,B,S}})=1_A (X^{\tau,\xi,B^1,S^1}_{\sigma^{\tau,\xi,B^1,S^1}},Y^{\tau,\xi,B^1,S^1}_{\sigma^{\tau,\xi,B^1,S^1}})+1_{A^c} (X^{\tau,\xi,B^2,S^2}_{\sigma^{\tau,\xi,B^2,S^2}},Y^{\tau,\xi,B^2,S^2}_{\sigma^{\tau,\xi,B^2,S^2}})\in \partial\mathcal{S}\cup\{\bfDelta\}\]
by the $(\tau,\xi)$-admissibility of $(B^i,S^i), i=1,2$.

(ii) Clearly, $\tau_1$ is a $\mathbb{G}$-stopping time taking values in $[\tau,\tau_d]$ and $\xi_1$ is $\mathcal{G}_{\tau_1}$-measurable. Since $\tau_1\leq \sigma^{\tau,\xi,B^1,S^1}$, the $(\tau,\xi)$-admissibility of $(B^1,S^1)$ implies $\xi_1\in\ubar{\mathcal{S}}\cup\{\bfDelta\}$. Moreover, $\xi_1=\bfDelta$ if and only if $\tau_1=\tau_d$. So $(\tau_1,\xi_1)$ is a valid random initial condition. It is routine to check $(B, S)\in\mathscr{A}_0$. To show $(B,S)\in\mathscr{A}(\tau,\xi)$, observe that 
\[(X^{\tau,\xi,B,S}_{t},Y^{\tau,\xi,B,S}_{t})=\begin{cases}
(X^{\tau,\xi,B^1,S^1}_{t},Y^{\tau,\xi,B^1,S^1}_{t}), & \tau\leq t<\tau_1,\\
(X^{\tau_1,\xi_1,B^2,S^2}_{t},Y^{\tau_1,\xi_1,B^2,S^2}_{t}), & t\geq \tau_1.
\end{cases}\]
This, together with $\tau_1\leq \sigma^{\tau,\xi,B^1,S^1}$, imply $\sigma^{\tau,\xi,B,S}=\sigma^{\tau_1,\xi_1,B^2,S^2}\geq \tau_1$.
Since $(B^2,S^2)\in\mathscr{A}(\tau_1,\xi_1)$, we have
\[(X^{\tau,\xi,B,S}_{\sigma^{\tau,\xi,B,S}},Y^{\tau,\xi,B,S}_{\sigma^{\tau,\xi,B,S}})=(X^{\tau_1,\xi_1,B^2,S^2}_{\sigma^{\tau_1,\xi_1,B^2,S^2}},Y^{\tau_1,\xi_1,B^2,S^2}_{\sigma^{\tau_1,\xi_1,B^2,S^2}})\in \partial\mathcal{S}\cup\{\bfDelta\}.\]
\end{proof}

\subsection{Main result}
\begin{thm}\label{thm:main result}
The value function $\psi$ is the unique (continuous) viscosity solution to the HJB equation \eqref{HJB} satisfying the boundary condition \eqref{BC}.
\end{thm}
The proof of Theorem \ref{thm:main result} is deferred to the end of Section~\ref{sec:comparison}.

\section{Stochastic supersolution}\label{sec:ssupsoln}

\begin{defn}\label{defn_ssupsoln}
A bounded u.s.c. function $v$ on $\ubar{\mathcal{S}}$ is called a \textit{stochastic supersolution} of \eqref{HJB}, \eqref{BC} if 
\begin{itemize}
\item[(SP1)] $v\geq 1$ on $\partial\mathcal{S}_b$, $v\geq 0$ on $\partial\mathcal{S}_{c/r}$;
\item[(SP2)] for any random initial condition $(\tau,\xi)$, there exists $(B,S)\in\mathscr{A}(\tau,\xi)$ 
such that
\begin{equation*}\label{supermart}
 \mathbb{E}[v(X^{\tau,\xi,B,S}_\rho, Y^{\tau,\xi,B,S}_\rho)|\mathcal{G}_\tau]\leq v(\xi)
\end{equation*}
for all $\mathbb{G}$-stopping time $\rho\in[\tau, \sigma^{\tau,\xi,B,S}]$, where $v$ is understood to be its extension to $\ubar{\mathcal{S}}\cup\{\bfDelta\}$.
\end{itemize}
Denote the set of stochastic supersolutions by $\mathcal{V}^+$. 
\end{defn}

\begin{remark}
$\mathcal{V}^+\neq\emptyset$ since the constant $1\in \mathcal{V}^+$. There is a more useful stochastic supersolution: the upper bound function $\ubar{\psi}$ defined in \eqref{upperbdd}, which satisfies (SP1) with equality. (See Lemma \ref{lemma:ubarpsi is a ssupsoln}.) The existence of such a stochastic supersolution 
automatically guarantees the the upper stochastic envelope (which will be introduced shortly) satisfies the boundary condition \eqref{BC}.
\end{remark}

\begin{remark}\label{rmk:ssupsoln>=psi}
Any stochastic supersolution $v$ dominates the value function $\psi$ on $\ubar{\mathcal{S}}$.
To see this, first note that $v\geq \psi$ on $\partial\mathcal{S}$ by (SP1). Then for any $(x,y)\in\mathcal{S}$, take $\tau=0$ and $\xi=(x,y)$. Let $(B,S)\in\mathscr{A}(x,y)$ be given by (SP2) for $v$. Let $\rho=\sigma^{x,y,B,S}$. To simplify notation, we write $\tau_b$ for $\tau_b^{x,y,B,S}$ and $\tau_s$ for $\tau_{s}^{x,y,B,S}:=\inf\{t\geq 0: (X^{x,y,B,S}_t, Y^{x,y,B,S}_t)\in\ubar{\mathcal{S}}_{c/r}\}$. We have
\[v(x,y)\geq \mathbb{E}\left[v(X^{x,y,B,S}_\rho, Y^{x,y,B,S}_\rho)\right]\geq\mathbb{E}\left[1_{\{(X^{x,y,B,S}_{\rho}, Y^{x,y,B,S}_{\rho})\in \partial\mathcal{S}_b\}}\right]=\mathbb{P}\left(\tau_b<\tau_d\wedge\tau_s\right).\]
where the first inequality holds by (SP2) and the second inequality holds by (SP1). Now, let
\[(B'_t, S'_t)=(B_t,S_t)1_{\{t<\tau_s\}}+((X^{x,y,B,S}_{\tau_s}-c/r)^++B_{\tau_s},(Y^{x,y,B,S}_{\tau_s})^++S_{\tau_s})1_{\{t\geq\tau_s\}}.\]
That is, $(B', S')$ follows $(B,S)$ before hitting the safe region, and at the moment when the safe region is hit (by diffusion), immediately liquidate all stock position and do no more transaction afterwards. This ensures that once the safe region is reached, death will definitely happen before ruin. It is easy to check $(B',S')\in\mathscr{A}_0$ and $\mathbb{P}(\tau_b<\tau_d\wedge \tau_s)=\mathbb{P}(\tau^{x,y,B',S'}_b<\tau_d)$. We therefore have
\[v(x,y)\geq \mathbb{P}(\tau^{x,y,B',S'}_b<\tau_d)\geq \psi(x,y).\]
\end{remark}

\begin{lemma}\label{lemma:ubarpsi is a ssupsoln}
$\ubar{\psi}\in\mathcal{V}^+$. 
\end{lemma}
\begin{proof}
We only show (SP2). Let $(\tau,\xi)$ be any random initial condition. Define
\[(B_t,S_t):=1_{\{t\geq \tau\}}\left(\frac{(\xi^1)^-}{1-\lambda}, (\xi^1)^+\right).\]
Intuitively, what $(B,S)$ does is to immediately liquidate the stock position at time $\tau$ and do no more transaction afterwards. It can be checked that $(X^{\tau,\xi,B,S}_{\sigma^{\tau,\xi,B,S}},Y^{\tau,\xi,B,S}_{\sigma^{\tau,\xi,B,S}})\in\{(b,0),(c/r,0),\bfDelta\}$, thus $(B,S)\in\mathscr{A}(\tau,\xi)$. We have $(X^{\tau,\xi,B,S}_\tau, Y^{\tau,\xi,B,S}_\tau)=1_{\{\tau<\tau_d\}}(L(\xi),0)+1_{\{\tau=\tau_d\}}\bfDelta$ and $Y^{\tau,\xi,B,S}_t=0$ for all $t\in[\tau,\tau_d)$. Let $\rho\in[\tau,\sigma^{\tau,\xi,B,S}]$ be any $\mathbb{G}$-stopping time. Let $f(x):=\ubar{\psi}(x,0)\in C[b,c/r]\cap C^2[b,c/r)$. With slight abuse of notation, we also write $X^{\tau,\xi,B,S}_t=\bfDelta$ when $t=\tau_d$, and set $f(\bfDelta)=0$. Apply It\^{o}'s formula to $f(X^{\tau,\xi,B,S})$, we get
\begin{align*}
f(X^{\tau,\xi,B,S}_\rho)-f(X^{\tau,\xi,B,S}_\tau)&=\int_\tau^\rho f'(X^{\tau,\xi,B,S}_t)(rX^{\tau,\xi,B,S}_t-c)dt+\int_\tau^\rho \left(f(\bfDelta)-f(X^{\tau,\xi,B,S}_{t-})\right)dN_t\\
&=\int_{\tau}^\rho \left[f'(x)(rx-c)-\beta f(x)\right]\big|_{x=X^{\tau,\xi,B,S}_t}dt+\int_\tau^\rho -f(X^{\tau,\xi,B,S}_{t-}) d(N_t-\beta t)\\
&=\int_\tau^\rho -f(X^{\tau,\xi,B,S}_{t-}) d(N_t-\beta t),
\end{align*}
where we used the explicit formula of $f$ to kill the drift. Taking conditional expectation yields
\[\mathbb{E}[f(X^{\tau,\xi,B,S}_\rho)|\mathcal{G}_\tau]=f(X^{\tau,\xi,B,S}_\tau).\]
It follows that
\begin{align*}
\mathbb{E}[\ubar{\psi}(X^{\tau,\xi,B,S}_\rho, Y^{\tau,\xi,B,S}_\rho)|\mathcal{G}_\tau]&=\mathbb{E}[1_{\{\rho<\tau_d\}}\ubar{\psi}(X^{\tau,\xi,B,S}_\rho, 0)|\mathcal{G}_\tau]=\mathbb{E}[1_{\{\rho<\tau_d\}}f(X^{\tau,\xi,B,S}_\rho)|\mathcal{G}_\tau]\\
&=\mathbb{E}[f(X^{\tau,\xi,B,S}_\rho)|\mathcal{G}_\tau]=f(X^{\tau,\xi,B,S}_\tau)=1_{\{\tau<\tau_d\}}f(L(\xi))\\
&=1_{\{\tau<\tau_d\}}\ubar{\psi}(L(\xi),0)=1_{\{\tau<\tau_d\}}\ubar{\psi}(\xi)=\ubar{\psi}(\xi).
\end{align*}
In the second last equality, we used $\ubar{\psi}(x,y)=\ubar{\psi}(L(x,y),0)$ for $(x,y)\in\ubar{\mathcal{S}}$.
\end{proof}

\begin{lemma}\label{lemma_stability_ssupsoln}
Let $v_1, v_2\in\mathcal{V}^+$. Then $v_1\wedge v_2\in\mathcal{V}^+$.
\end{lemma}
\begin{proof}
The minimum of bounded u.s.c. functions is still bounded and u.s.c.. (SP1) is clearly satisfied. For (SP2), let $(B^i, S^i)\in\mathscr{A}(\tau,\xi), i=1,2$ be the admissible control corresponding to $v_i$ and the random initial condition $(\tau,\xi)$. Put $A:=\{v_1(\xi)\leq v_2(\xi)\}\in\mathcal{G}_{\tau}$. The control
\[(B_t,S_t):=1_{\{t\geq\tau\}}\left[\left(B^1_t-B^1_{\tau-},S^1_t-S^1_{\tau-}\right)1_A+\left(B^2_t-B^2_{\tau-},S^2_t-S^2_{\tau-}\right)1_{A^c}\right]\]
serves the purpose. $(\tau,\xi)$-admissible follows from Lemma \ref{lemma:admissibility}.i, and the remaining proof is very similar to that of \cite[Lemma 1]{Rokhlin14a} except that the process $Z$ is replaced by $v(X,Y)$ and the direction of inequalities are reversed. So we omit the details.
\end{proof}

\begin{prop}\label{prop:v+ is a vsubsoln}
The \textit{upper stochastic envelope}
\[v_+(x,y):=\inf_{v\in \mathcal{V}^+} v(x,y)\]
is a viscosity subsolution of \eqref{HJB} satisfying $v_+\leq 1$ on $\partial \mathcal{S}_b$ and $v_+\leq 0$ on $\partial \mathcal{S}_{c/r}$. 
\end{prop}
\begin{proof}
The boundary inequalities are satisfied because $v_+\leq \ubar{\psi}$ by Lemma \ref{lemma:ubarpsi is a ssupsoln}.\footnote{In fact, equalities hold for $v_+$ on the boundary; the reverse inequalities come from the simple fact that (SP1) is preserved under pointwise infimum.} 
To show interior viscosity subsolution property, let $(x_0, y_0)\in \mathcal{S}$ and $\varphi\in C^2(\mathcal{S})$ be a test function such that $v_+-\varphi$ attains a strict local maximum of zero at $(x_0, y_0)$. We need to show
\[\max\left\{\mathcal{L}\varphi, -(1-\mu)\varphi_x+\varphi_y, \varphi_x-(1-\lambda)\varphi_y\right\}(x_0, y_0)\leq 0.\]
Assume on the contrary that 
\[\max\left\{\mathcal{L}\varphi, -(1-\mu)\varphi_x+\varphi_y, \varphi_x-(1-\lambda)\varphi_y\right\}(x_0, y_0)> 0.\]
There are three cases to consider: (i) $\mathcal{L}\varphi(x_0, y_0)>0$, (ii) $-(1-\mu)\varphi_x(x_0, y_0)+\varphi_y(x_0, y_0)>0$, and (iii) $\varphi_x(x_0, y_0)-(1-\lambda)\varphi_y(x_0, y_0)>0$. We will show that each case leads to a contradiction.

\textbf{Case (i)}. $\mathcal{L}\varphi(x_0, y_0)>0$. We can find, by continuity, a small closed ball $\ubar{B_{\epsilon}(x_0, y_0)}\subseteq\mathcal{S}$ such that 
\begin{equation*}
\mathcal{L}\varphi>0 \quad \text{on} \quad \ubar{B_{\epsilon}(x_0, y_0)}.
\end{equation*}
Since $v_+-\varphi$ is u.s.c. and $\ubar{B_{\epsilon}(x_0, y_0)}\backslash B_{\epsilon/2}(x_0, y_0)$ is compact, there exists a $\delta>0$ such that 
\[v_+-\varphi\leq -\delta \quad \text{on}\quad \ubar{B_{\epsilon}(x_0, y_0)}\backslash B_{\epsilon/2}(x_0, y_0).\]
By \cite[Proposition 4.1]{BayraktarSirbu12} and Lemma \ref{lemma_stability_ssupsoln}, $v_+$ can be approximated from above by a non-increasing sequence of stochastic supersolutions $v_n$. By \cite[Lemma 2.4]{BayraktarSirbu14}, there exists a large enough $N$ such that $v:=v_N$ satisfies
\[v-\varphi\leq -\frac{\delta}{2}\quad \text{on }\ubar{B_\epsilon(x_0, y_0)}\backslash B_{\epsilon/2}(x_0, y_0).\]
Choose $\eta\in(0,\delta/2)$ small so that $\varphi^\eta:=\varphi-\eta$ satisfies 
\begin{equation}\label{ssupsoln_eq1}
\mathcal{L}\varphi^\eta>0 \quad \text{on } \ubar{B_{\epsilon}(x_0, y_0)}.
\end{equation}
We also have
\begin{equation}\label{ssupsoln_eq2}
v\leq {\varphi}-\frac{\delta}{2}<{\varphi}-\eta=\varphi^\eta \quad \text{on }B_\epsilon(x_0, y_0)\backslash B_{\epsilon/2}(x_0, y_0),
\end{equation}
and
\begin{equation}\label{ssupsoln_eq3}
\varphi^\eta(x_0,y_0)={\varphi}(x_0,y_0)-\eta=v_+(x_0,y_0)-\eta<v_+(x_0,y_0).
\end{equation}
Define
\[v^\eta:=\begin{cases}
v\wedge \varphi^\eta &\text{ on } \ubar{B_\epsilon(x_0,y_0)},\\
v &\text{ on } \ubar{B_\epsilon(x_0,y_0)}^c.
\end{cases}\]
If we can show $v^\eta\in\mathcal{V}^+$, then \eqref{ssupsoln_eq3} will lead to a contradiction to the (pointwise) minimality of $v_+$. Clearly, $v^\eta$ is u.s.c. since the minimum of u.s.c. functions is u.s.c. and $v^\eta=v$ outside $B_{\epsilon/2}(x_0, y_0)$. Boundedness is also easy.
(SP1) is satisfied because $v^\eta=v$ on $\partial \mathcal{S}$. The remaining proof of case (i) is devoted to the verification of (SP2), i.e. the supermartingale property.

Let $(\tau,\xi)$ be any random initial condition and $(B^0,S^0)$ be the $(\tau,\xi)$-admissible control in (SP2) for the stochastic supersolution $v$. Let
\[A:=\{\xi\in B_{\epsilon/2}(x_0,y_0)\}\cap\{\varphi^\eta(\xi)<v(\xi)\}\in\mathcal{G}_\tau.\]
Define a new control
\[(B^1_t,S^1_t):=1_{A^c\cap\{t\geq\tau\}}(B^0_t-B^0_{\tau-},S^0_t-S^0_{\tau-}).\]
$(B^1, S^1)$ follows $(B^0, S^0)$ starting from time $\tau$ when the position $\xi$ satisfies $v^\eta(\xi)=v(\xi)$, i.e. when it is optimal to use the control corresponding to $v$. By Lemma \ref{lemma:admissibility}.i, $(B^1,S^1)\in\mathscr{A}(\tau,\xi)$. Let
\[\tau_1:=\inf\{t\in[\tau,\sigma^{\tau,\xi,B^1,S^1}]: (X^{\tau,\xi,B^1,S^1}_t, Y_t^{\tau,\xi,B^1,S^1})\notin B_{\epsilon/2}(x_0, y_0)\}\]
be the exit time of the ball $B_{\epsilon/2}(x_0,y_0)$ and
\[\xi_1:=(X^{\tau,\xi,B^1,S^1}_{\tau_1}, Y_{\tau_1}^{\tau,\xi,B^1,S^1})\in\mathcal{G}_{\tau_1}\]
be the exit position. Since $X^{\tau,\xi,B^1, S^1}$ and $Y^{\tau,\xi,B^1, S^1}$ are RCLL, we have $\xi_1\notin B_{\epsilon/2}(x_0, y_0)$.\footnote{If $\xi\notin B_{\epsilon/2}(x_0, y_0)$, it is possible for the process to immediately jump back to $B_{\epsilon/2}(x_0, y_0)$ at time $\tau$. In this case, although we start outside the ball, $\tau_1\neq\tau$ because $(X^{\tau,\xi,B^1,S^1}_{t},Y^{\tau,\xi,B^1,S^1}_{t})$ gives the post-jump position at time $t$ which is inside the ball at $t=\tau$, and will stay inside the ball for some positive amount of time by the right continuity of its paths.} By Lemma \ref{lemma:admissibility}.ii, $(\tau_1,\xi_1)$ is a valid random initial condition. Let $(B^2, S^2)$ be the $(\tau_1,\xi_1)$-admissible control in (SP2) for $v$. Set
\[(B_t,S_t):=(B^1_t,S^1_t)1_{\{t<\tau_1\}}+(B^2_t-B^2_{\tau_1-}+B^1_{\tau_1},S^2_t-S^2_{\tau_1-}+S^1_{\tau_1})1_{\{t\geq \tau_1\}}.\]
Note that we allow ``double transactions" at time $\tau_1$, first by $(\triangle B^1_{\tau_1}, \triangle S^1_{\tau_1})$, then by $(\triangle B^2_{\tau_1}, \triangle S^2_{\tau_1})$.
Lemma \ref{lemma:admissibility}.ii also implies $(B,S)\in\mathscr{A}(\tau,\xi)$. We now check the supermatingale property (SP2) for $v^\eta$ with control $(B,S)$.

Let $\rho$ be any $\mathbb{G}$-stopping time taking values in $[\tau,\sigma^{\tau,\xi,B,S}]$. In the event $A$, $(B^1, S^1)=0$ so that $(X^{\tau,\xi,B^1, S^1}, Y^{\tau,\xi,B^1,S^1})$ exits $B_{\epsilon/2}(x_0, y_0)$ either by diffusion or by death, giving $\xi_1\in \partial B_{\epsilon/2}(x_0, y_0)\cup \{\bfDelta\}$. The control $(B,S)$ is inactive before time $\tau_1$ and equals $(\triangle B^2_{\tau_1}, \triangle S^2_{\tau_1})$ at $\tau_1$. By It\^{o}'s formula, we have in the event $A$
\begin{align*}
&\varphi^\eta(X^{\tau,\xi,B,S}_{\rho\wedge \tau_1}, Y^{\tau,\xi,B,S}_{\rho\wedge \tau_1})-\varphi^\eta(X^{\tau, \xi, B,S}_\tau, Y^{\tau,\xi, B,S}_\tau)\\
&=\int_{\tau}^{\rho\wedge \tau_1} -\mathcal{L}\varphi^\eta(X^{\tau,\xi,B,S}_t,Y^{\tau,\xi,B,S}_t)dt+\int_\tau^{\rho\wedge \tau_1}(\varphi^\eta)'(X^{\tau,\xi,B,S}_t, Y^{\tau,\xi,B,S}_t)\sigma Y^{\tau,\xi,B,S}_t dW_t\\
&\quad+\int_\tau^{\rho\wedge \tau_1}\left[\varphi^\eta(\bfDelta)-\varphi^\eta(X^{\tau,\xi,B,S}_{t-}, Y^{\tau,\xi,B,S}_{t-})\right]d(N_t-\beta t)+1_{\{\rho\geq\tau_1\}}\left[\varphi^\eta(\xi_1+\triangle \xi)-\varphi^\eta(\xi_1)\right],
\end{align*}
where
\[\triangle \xi:=(-1,1-\lambda)\triangle B^2_{\tau_1}+(1-\mu,-1)\triangle S^2_{\tau_1}.\]
Since $(X^{\tau,\xi,B,S}_t, Y^{\tau,\xi,B,S}_t)\in B_{\epsilon/2}(x_0, y_0)$ for $\tau\leq t<\tau_1$ on $A$, and $\mathcal{L}\varphi^\eta>0$ in $B_{\epsilon/2}(x_0, y_0)$ by \eqref{ssupsoln_eq1}, the $dt$-integral is non-positive. The integrals with respect to the Brownian motion and the compensated Poisson process vanish by taking $\mathcal{G}_\tau$-conditional expectation. We therefore obtain
\begin{align*}
&\mathbb{E}[ 1_{A}\varphi^\eta (X^{\tau,\xi,B,S}_{\rho\wedge \tau_1}, Y^{\tau,\xi,B,S}_{\rho \wedge \tau_1})-1_{A\cap\{\rho\geq\tau_1\}}(\varphi^\eta(\xi_1+\triangle \xi)-\varphi^\eta(\xi_1))|\mathcal{G}_{\tau}]\\
&\leq 1_A\varphi^\eta(X^{\tau, \xi, B,S}_\tau, Y^{\tau,\xi, B,S}_\tau)=1_{A\cap\{\tau<\tau_d\}}\varphi^\eta(\xi)+1_{A\cap\{\tau=\tau_d\}}\varphi^\eta(\bfDelta)\\
&=1_A\varphi^\eta(\xi)\leq 1_Av^\eta(\xi).
\end{align*}
In the last equality, we used $\xi=\bfDelta$ if $\tau=\tau_d$. Notice that 
\begin{align*}
&1_{A}\varphi^\eta (X^{\tau,\xi,B,S}_{\rho\wedge \tau_1}, Y^{\tau,\xi,B,S}_{\rho \wedge \tau_1})-1_{A\cap\{\rho\geq\tau_1\}}(\varphi^\eta(\xi_1+\triangle \xi)-\varphi^\eta(\xi_1))\\
&=1_{A\cap\{\rho<\tau_1\}}\varphi^\eta (X^{\tau,\xi,B,S}_{\rho}, Y^{\tau,\xi,B,S}_{\rho})+1_{A\cap\{\rho\geq \tau_1\}}\varphi^\eta (\xi_1+\triangle \xi)-1_{A\cap\{\rho\geq\tau_1\}}(\varphi^\eta(\xi_1+\triangle \xi)-\varphi^\eta(\xi_1))\\
&=1_{A\cap\{\rho<\tau_1\}}\varphi^\eta (X^{\tau,\xi,B,S}_{\rho}, Y^{\tau,\xi,B,S}_{\rho})+1_{A\cap\{\rho\geq\tau_1\}}\varphi^\eta(\xi_1).
\end{align*}
So 
\[
\mathbb{E}[1_{A\cap\{\rho<\tau_1\}}\varphi^\eta (X^{\tau,\xi,B,S}_{\rho}, Y^{\tau,\xi,B,S}_{\rho})+1_{A\cap\{\rho\geq\tau_1\}}\varphi^\eta(\xi_1)|\mathcal{G}_{\tau}]\leq 1_A v^\eta(\xi).
\]
We have argued that $\xi_1\in\partial B_{\epsilon/2}(x_0, y_0)\cup\{\bfDelta\}$ on $A$. By \eqref{ssupsoln_eq2} and the definition of $v^\eta$, we know $v^\eta\leq \varphi^\eta$ in $B_{\epsilon}(x_0, y_0)$. This allows us to replace $\varphi^\eta$ by $v^\eta$ in the above inequality and get
\begin{equation}\label{ssupsoln_eq5}
\mathbb{E}[1_{A\cap\{\rho<\tau_1\}}v^\eta(X^{\tau,\xi,B,S}_{\rho}, Y^{\tau,\xi,B,S}_{\rho})+1_{A\cap\{\rho\geq\tau_1\}}v^\eta(\xi_1)|\mathcal{G}_{\tau}]\leq 1_Av^\eta(\xi).
\end{equation}

By ``optimality" of $(B^0, S^0)$ (and thus $(B^1, S^1)$ on $A^c$) for $v$ with random initial condition $(\tau,\xi)$, we have
\[\mathbb{E}[1_{A^c} v(X^{\tau,\xi, B^1, S^1}_{\rho\wedge \tau_1}, Y^{\tau,\xi, B^1, S^1}_{\rho\wedge \tau_1})|\mathcal{G}_{\tau}]\leq 1_{A^c}v(\xi).\]
Since $v^\eta\leq v$ everywhere, we can replace $v$ by $v^\eta$ on the left hand side in the above inequality. On $A^c$, either $\xi\notin B_{\epsilon/2}(x_0, y_0)$, or $\xi\in B_{\epsilon/2}(x_0, y_0)$ and $v(\xi)\leq \varphi^\eta(\xi)$. In both cases, $v(\xi)=v^\eta(\xi)$ since $v^\eta=v$ outside the ball $B_{\epsilon/2}(x_0, y_0)$. So we can also replace $v$ by $v^\eta$ on the right hand side. Splitting the set $A^c$ on the left hand side according to the relation between $\rho$ and $\tau_1$, and using the definition of $(B,S)$, we have
\begin{equation}\label{ssupsoln_eq6}
\mathbb{E}[1_{A^c\cap\{\rho<\tau_1\}} v^\eta(X^{\tau,\xi, B, S}_{\rho}, Y^{\tau,\xi, B, S}_{\rho})+1_{A^c\cap\{\rho\geq\tau_1\}} v^\eta(\xi_1)|\mathcal{G}_{\tau}]\leq 1_{A^c}v^\eta(\xi).
\end{equation}
Combining \eqref{ssupsoln_eq5} and \eqref{ssupsoln_eq6} gives us
\begin{equation}\label{ssupsoln_eq7}
\mathbb{E}[1_{\{\rho<\tau_1\}} v^\eta(X^{\tau,\xi, B, S}_{\rho}, Y^{\tau,\xi, B, S}_{\rho})+1_{\{\rho\geq\tau_1\}} v^\eta(\xi_1)|\mathcal{G}_{\tau}]\leq v^\eta(\xi).
\end{equation}

By ``optimality" of $(B^2, S^2)$ for $v$ with random initial condition $(\tau_1,\xi_1)$, we have (by applying the supermartingale property to the stopping time $\rho\vee \tau_1$)
\[\mathbb{E}[1_{\{\rho\geq\tau_1\}} v(X^{\tau_1,\xi_1, B^2, S^2}_{\rho}, Y^{\tau_1,\xi_1, B^2, S^2}_{\rho})|\mathcal{G}_{\tau_1}]\leq 1_{\{\rho\geq\tau_1\}}v(\xi_1).\]
Same as before, we can replace all $v$'s by $v^\eta$ in the above inequality because $v^\eta\leq v$ everywhere, $v=v^\eta$ outside $B_{\epsilon/2}(x_0, y_0)$ and $\xi_1$, being the exit position, is outside $B_{\epsilon/2}(x_0, y_0)$. So
\begin{align*}
\mathbb{E}[1_{\{\rho\geq\tau_1\}} v^\eta(X^{\tau,\xi, B, S}_{\rho}, Y^{\tau,\xi, B, S}_{\rho})|\mathcal{G}_{\tau_1}]&=\mathbb{E}[1_{\{\rho\geq\tau_1\}} v^\eta(X^{\tau_1,\xi_1, B^2, S^2}_{\rho}, Y^{\tau_1,\xi_1, B^2, S^2}_{\rho})|\mathcal{G}_{\tau_1}]\leq 1_{\{\rho\geq\tau_1\}}v^\eta(\xi_1).
\end{align*}
Taking $\mathcal{G}_{\tau}$-condition expectation and using tower property yields
\begin{equation}\label{ssupsoln_eq8}
\mathbb{E}[1_{\{\rho\geq\tau_1\}} v^\eta(X^{\tau,\xi, B, S}_{\rho}, Y^{\tau,\xi, B, S}_{\rho})-1_{\{\rho\geq\tau_1\}}v^\eta(\xi_1)|\mathcal{G}_{\tau}]\leq 0.
\end{equation}
Finally, we add \eqref{ssupsoln_eq7} and \eqref{ssupsoln_eq8} to get
\[\mathbb{E}[v^\eta(X^{\tau,\xi, B, S}_{\rho}, Y^{\tau,\xi, B, S}_{\rho})|\mathcal{G}_{\tau}]\leq v^\eta(\xi).\]
This completes the proof of (SP2) and hence of case (i).

\textbf{Case (ii).} $-(1-\mu)\varphi_x(x_0, y_0)+\varphi_y(x_0, y_0)>0$. The proof is in most part similar to that of case (i). So we shall be brief on the similar parts. Same as in case (i), we can find $\epsilon, \eta>0$ and $v\in \mathcal{V}^+$ such that $\varphi^\eta:=\varphi-\eta$ satisfies
\begin{equation}\label{ssupsoln_ii_eq1}
-(1-\mu)\varphi^\eta_x+\varphi^\eta_y>0 \quad \text{on} \quad \ubar{B_{\epsilon}(x_0, y_0)},
\end{equation}
\begin{equation*}
v\leq \varphi^\eta \quad \text{on} \quad \ubar{B_\epsilon(x_0, y_0)}\backslash B_{\epsilon/2}(x_0, y_0),
\end{equation*}
\begin{equation*}
 \varphi^\eta(x_0, y_0)<v_+(x_0, y_0).
\end{equation*}
Define
\[v^\eta:=\begin{cases}
v\wedge \varphi^\eta &\text{ on } \ubar{B_\epsilon(x_0,y_0)},\\
v &\text{ on } \ubar{B_\epsilon(x_0,y_0)}^c.
\end{cases}\]
It suffices to show $v^\eta\in\mathcal{V}^+$. And the only nontrivial part is to check $v^\eta$ satisfies (SP2). 

Let $(\tau,\xi)$ be any random initial condition and $(B^0,S^0)$ be a $(\tau,\xi)$-admissible control in (SP2) for the stochastic supersolution $v$. Let
\[A:=\{\xi\in B_{\epsilon/2}(x_0,y_0)\}\cap\{\varphi^\eta(\xi)<v(\xi)\}\in\mathcal{G}_\tau.\]
Observe that \eqref{ssupsoln_ii_eq1} implies for any $(x,y)\in B_{\epsilon}(x_0, y_0)$ and $h>0$ small such that $(x+(1-\mu)h, y-h)\in B_{\epsilon}(x_0, y_0)$, we have 
\begin{equation}\label{ssupsoln_ii_eq2}
\varphi^\eta(x+(1-\mu)h, y-h)-\varphi^\eta(x,y)=h [(1-\mu)\varphi^\eta_x-\varphi^\eta_y](x+(1-\mu)h', y-h')<0
\end{equation}
for some $h'\in(0,h)$ by Mean Value Theorem. This suggests selling stocks is optimal on the set $A$. Given a point $(x,y)\in B_{\epsilon/2}(x_0, y_0)$, denote by $\mathfrak{s}(x,y)=(\mathfrak{s}^0(x,y), \mathfrak{s}^1(x,y))$ the intersection of the ray $\{(x+(1-\mu)h, y-h): h\geq 0\}$ and $\partial B_{\epsilon/2}(x_0, y_0)$, i.e. the unique point on $\partial B_{\epsilon/2}(x_0, y_0)$ that can be reached by a sell. Define a new control
\[(B^1_t,S^1_t):=1_{A\cap\{t\geq\tau\}}(0,\xi^1-\mathfrak{s}^1(\xi))+1_{A^c\cap\{t\geq\tau\}}(B^0_t-B^0_{\tau-},S^0_t-S^0_{\tau-}).\]
$(B^1, S^1)$ says starting at time $\tau$, if we are in $A$, we immediately jump to $\partial B_{\epsilon/2}(x_0, y_0)$ by a sell and do nothing afterwards; if we are in $A^c$, we follow $(B^0, S^0)$. A slight variation of Lemma \ref{lemma:admissibility}.i shows $(B^1,S^1)$ is $(\tau,\xi)$-admissible. Let
\[\tau_1:=\inf\{t\in[\tau,\sigma^{\tau,\xi,B^1,S^1}]: (X^{\tau,\xi,B^1,S^1}_t, Y_t^{\tau,\xi,B^1,S^1})\notin B_{\epsilon/2}(x_0, y_0)\}\]
be the exit time of the ball $B_{\epsilon/2}(x_0,y_0)$ and
\[\xi_1:=(X^{\tau,\xi,B^1,S^1}_{\tau_1}, Y_{\tau_1}^{\tau,\xi,B^1,S^1})\in\mathcal{G}_{\tau_1}\]
be the exit position. As in case (i), $\xi_1\notin B_{\epsilon/2}(x_0, y_0)$ and $(\tau_1,\xi_1)$ is a valid random initial condition. Also notice that on $A$, $\tau_1=\tau$ and $\xi_1=\mathfrak{s}(\xi)$ if $\tau<\tau_d$. 
Let $(B^2, S^2)$ be a $(\tau_1,\xi_1)$-admissible control in (SP2) for $v$. Set
\[(B_t,S_t):=(B^1_t,S^1_t)1_{\{ t<\tau_1\}}+(B^2_t-B^2_{\tau_1-}+B^1_{\tau_1},S^2_t-S^2_{\tau_1-}+S^1_{\tau_1})1_{\{t\geq \tau_1\}}.\]
$(B,S)\in\mathscr{A}(\tau,\xi)$ by Lemma \ref{lemma:admissibility}.ii. It remains to check (SP2) for $v^\eta$ with control $(B,S)$.

Let $\rho$ be any $\mathbb{G}$-stopping time taking values in $[\tau,\sigma^{\tau,\xi,B,S}]$. In the event $A$ (recall that $\tau_1=\tau$), when $\tau<\tau_d$, \eqref{ssupsoln_ii_eq2} implies $\varphi^\eta(\xi_1)=\varphi^\eta(\mathfrak{s}(\xi))<\varphi^\eta(\xi)$; when $\tau=\tau_d$, $\varphi^\eta(\xi_1)=\varphi^\eta(\xi)=\varphi^\eta(\bfDelta)=0$. So
\begin{equation}\label{ssupsoln_ii_eq3}
1_{A} v^\eta(\xi_1) \leq 1_{A}\varphi^\eta(\xi_1)<1_{A}\varphi^\eta(\xi)=1_{A} v^\eta(\xi).
\end{equation}
In the event $A^c$, we use that $(B^0, S^0)$ is ``optimal" for $v$ to obtain
\begin{equation}\label{ssupsoln_ii_eq4}
\begin{aligned}
&\mathbb{E}[1_{A^c\cap\{\rho<\tau_1\}} v^\eta(X^{\tau,\xi, B^1, S^1}_{\rho}, Y^{\tau,\xi, B^1, S^1}_{\rho})+1_{A^c\cap\{\rho\geq \tau_1\}}v^\eta (\xi_1)|\mathcal{G}_{\tau}]\\
&=\mathbb{E}[1_{A^c} v^\eta(X^{\tau,\xi, B^1, S^1}_{\rho\wedge \tau_1}, Y^{\tau,\xi, B^1, S^1}_{\rho\wedge \tau_1})|\mathcal{G}_{\tau}]\leq\mathbb{E}[1_{A^c} v(X^{\tau,\xi, B^1, S^1}_{\rho\wedge \tau_1}, Y^{\tau,\xi, B^1, S^1}_{\rho\wedge \tau_1})|\mathcal{G}_{\tau}]\\
&=\mathbb{E}[1_{A^c} v(X^{\tau,\xi, B^0, S^0}_{\rho\wedge \tau_1}, Y^{\tau,\xi, B^0, S^0}_{\rho\wedge \tau_1})|\mathcal{G}_{\tau}]\leq 1_{A^c}v(\xi)=1_{A^c}v^\eta(\xi).
\end{aligned}
\end{equation}
Combining \eqref{ssupsoln_ii_eq3} and \eqref{ssupsoln_ii_eq4}, and using that $(B, S)$ equals $(B^1, S^1)$ on $[\tau,\tau_1)$, we get
\begin{equation}\label{ssupsoln_ii_eq5}
\mathbb{E}[1_{\{\rho<\tau_1\}} v^\eta(X^{\tau,\xi, B, S}_{\rho}, Y^{\tau,\xi, B, S}_{\rho})+1_{\{\rho\geq \tau_1\}}v^\eta (\xi_1)|\mathcal{G}_{\tau}]\leq v^\eta(\xi).
\end{equation}
By ``optimality" of $(B^2, S^2)$ for $v$ with random initial condition $(\tau_1,\xi_1)$, we have
\begin{equation*}
\begin{aligned}
\mathbb{E}[1_{\{\rho\geq\tau_1\}} v^\eta(X^{\tau,\xi, B, S}_{\rho}, Y^{\tau,\xi, B, S}_{\rho})|\mathcal{G}_{\tau_1}]&=\mathbb{E}[1_{\{\rho\geq\tau_1\}} v^\eta(X^{\tau_1,\xi_1, B^2, S^2}_{\rho}, Y^{\tau_1,\xi_1, B^2, S^2}_{\rho})|\mathcal{G}_{\tau_1}]\\
&\leq\mathbb{E}[1_{\{\rho\geq\tau_1\}} v(X^{\tau_1,\xi_1, B^2, S^2}_{\rho}, Y^{\tau_1,\xi_1, B^2, S^2}_{\rho})|\mathcal{G}_{\tau_1}]\\
&\leq 1_{\{\rho\geq\tau_1\}}v(\xi_1)=1_{\{\rho\geq \tau_1\}}v^\eta(\xi_1).
\end{aligned}
\end{equation*}
Taking $\mathcal{G}_{\tau}$-condition expectation yields
\begin{equation}\label{ssupsoln_ii_eq6}
\mathbb{E}[1_{\{\rho\geq\tau_1\}} v^\eta(X^{\tau,\xi, B, S}_{\rho}, Y^{\tau,\xi, B, S}_{\rho})-1_{\{\rho\geq\tau_1\}}v^\eta(\xi_1)|\mathcal{G}_{\tau}]\leq 0.
\end{equation}
Finally, we add \eqref{ssupsoln_ii_eq5} and \eqref{ssupsoln_ii_eq6} to get
\[\mathbb{E}[v^\eta(X^{\tau,\xi, B, S}_{\rho}, Y^{\tau,\xi, B, S}_{\rho})|\mathcal{G}_{\tau}]\leq v^\eta(\xi).\]
This completes the proof of case (ii).

\textbf{Case (iii).} $\varphi_x(x_0, y_0)-(1-\lambda)\varphi_y(x_0, y_0)>0$. This case is symmetric to case (ii). Buying stock is optimal in a neighborhood of $(x_0, y_0)$. We define the set $A$ and the ``optimal" $(\tau,\xi)$-admissible control in the same way as in case (ii) except one modification: in the definition of $(B^1, S^1)$, $(0,\xi^1-\mathfrak{s}^1(\xi))$ is replaced by $(\xi^0-\mathfrak{b}^0(\xi),0)$, where for $(x,y)\in B_{\epsilon/2}(x,y)$, $\mathfrak{b}(x,y)$ is defined to be the intersection of the ray $\{x-h, y+(1-\lambda)h: h\geq 0\}$ and $\partial B_{\epsilon/2}(x_0, y_0)$, i.e. the unique point on $\partial B_{\epsilon/2}(x_0, y_0)$ that can be reached by a buy. The rest of the argument is almost the same.
\end{proof}

\section{Stochastic subsolution}\label{sec:ssubsoln}

\begin{defn}\label{defn_ssubsoln}
A bounded l.s.c. function $v$ on $\ubar{\mathcal{S}}$ is called a \textit{stochastic subsolution} of \eqref{HJB}, \eqref{BC} if 
\begin{itemize}
\item[(SB1)] $v\leq 1$ on $\partial\mathcal{S}_b$, $v\leq 0$ on $\partial\mathcal{S}_{c/r}$;
\item[(SB2)] for any random initial condition $(\tau,\xi)$, control pair $(B,S)\in\mathscr{A}(\tau,\xi)$ and $\mathbb{G}$-stopping time $\rho\in[\tau, \sigma^{\tau,\xi,B,S}]$,
\begin{equation*}\label{supermart}
 \mathbb{E}[v(X^{\tau,\xi,B,S}_\rho, Y^{\tau,\xi,B,S}_\rho)|\mathcal{G}_\tau]\geq v(\xi),
\end{equation*}
where $v$ is understood to be its extension to $\ubar{\mathcal{S}}\cup\{\bfDelta\}$.
\end{itemize}
Denote the set of stochastic subsolutions by $\mathcal{V}^-$.
\end{defn}

\begin{remark}
$\mathcal{V}^-\neq\emptyset$ since the constant $0\in\mathcal{V}^-$. Similar to the stochastic supersolution case, there is also a  member of $\mathcal{V}^-$ which satisfies $(SB1)$ with equalities, namely, the lower bound function $\lbar{\psi}$ defined in \eqref{lowerbdd}. (See Lemma \ref{lemma:lbarpsi is a ssubsoln}.)
\end{remark}

\begin{remark}\label{rmk:ssubsoln<=psi}
Any stochastic subsolution $v$ is dominated by the value function $\psi$ on $\ubar{\mathcal{S}}$. Indeed, on $\partial\mathcal{S}$, we clearly have $v\leq \psi$ by (SP1). For $(x,y)\in\mathcal{S}$, take $\tau=0$, $\xi=(x,y)$, $(B,S)$ be any $(x,y)$-admissible control, and $\rho=\sigma^{x,y,B,S}$. We have by (SB2) and (SB1) that
\begin{align*}
v(x,y)&\leq \mathbb{E}[v(X^{x,y,B,S}_\rho, Y^{x,y,B,S}_\rho)]\leq \mathbb{E}\left[1_{\{\rho= \tau_b^{x,y,B,S}\}}\right]\\
&=\mathbb{P}(\tau_b^{x,y,B,S}<\tau_d\wedge \tau_{s}^{x,y,B,S})\leq \mathbb{P}(\tau_b^{x,y,B,S}<\tau_d).
\end{align*}
 Since this holds for any $(B,S)\in\mathscr{A}(x,y)$, taking infimum yields
 \[v(x,y)\leq \inf_{(B,S)\in\mathscr{A}(x,y)} \mathbb{P}(\tau_b^{x,y,B,S}<\tau_d)=\psi(x,y). \]
\end{remark}

\begin{lemma}\label{lemma:stability_ssubsoln}
Let $v_1, v_2\in\mathcal{V}^-$. Then $v_1\vee v_2\in\mathcal{V}^-$.
\end{lemma}
\begin{proof}
The maximum of bounded l.s.c. functions is still bounded and l.s.c.. (SB1) is clearly stable under maximum. For (SB2), simply notice that
\[\mathbb{E}[(v_1\vee v_2)(X^{\tau,\xi,B,S}_\rho, Y^{\tau,\xi,B,S}_\rho)|\mathcal{G}_\tau]\geq \mathbb{E}[v_i(X^{\tau,\xi,B,S}_\rho, Y^{\tau,\xi,B,S}_\rho)|\mathcal{G}_\tau]\geq v_i(\xi), \ i=1,2.\]
So
\[\mathbb{E}[(v_1\vee v_2)(X^{\tau,\xi,B,S}_\rho, Y^{\tau,\xi,B,S}_\rho)|\mathcal{G}_\tau]\geq (v_1\vee v_2)(\xi).\]
\end{proof}
\begin{remark}\label{rmk:stability_ssubsoln}
The above proof can be easily generalized to the countable case. In particular, the supremum of a countable family of stochastic subsolutions is bounded from above because every stochastic subsolution is dominated by the value function. In fact, it also generalizes to the uncountable case by \cite[Proposition 4.1]{BayraktarSirbu12} which says the supremum of an uncountable family of l.s.c. functions equals the supremum over some countable subfamily.
\end{remark}

\begin{lemma}\label{lemma:lbarpsi is a ssubsoln}
$\lbar{\psi}\in\mathcal{V}^-$.
\end{lemma}
\begin{proof}
Recall that $\lbar{\psi}$ can be written as the supremum of all $\psi_k$'s with $k\in(1-\mu,\frac{1}{1-\lambda})\cap\mathbb{Q}$ where $\psi_k$ is defined in \eqref{eq:psi_k}. To show $\lbar{\psi}\in\mathcal{V}^-$, it suffices to show $\psi_k\in\mathcal{V}^-$ for $k\in(1-\mu,\frac{1}{1-\lambda})$ by Lemma \ref{lemma:stability_ssubsoln} and the remark after it. To see (SB2) holds for $\psi_k$, let $(\tau,\xi)$ be any random initial condition, $(B,S)$ be any $(\tau,\xi)$-admissible control and $\rho\in[\tau, \sigma^{\tau,\xi,B,S}]$ be any $\mathbb{G}$-stopping time. For brevity, we shall omit the superscripts $(\tau,\xi,B,S)$ in all controlled processes and relevant stopping times in the rest of this proof. For functions defined on $[b,\infty)$, we extend them to $[b,\infty)\cup \{\bfDelta\}$ by assigning zero to the function value at $\bfDelta$. Define a new process
\[Z_t:=\begin{cases}
X_t+kY_t,& t<\tau_d,\\
\bfDelta, & t\geq \tau_d.
\end{cases}\]
Observe that $Z_t\in [b,\infty)\cup \{\bfDelta\}$ for all $t\in[\tau,\rho]$. We also have
\[dZ_t=(rZ_t+(\alpha-r) kY_t -c)dt+\sigma k Y_t dW_t+[k(1-\lambda)-1]dB_t+(1-\mu-k)dS_t, \quad t<\tau_d.\]
Since $1-\mu<k< \frac{1}{1-\lambda}$, the $dB$ and $dS$ terms are non-positive. 
So for $t\geq \tau$, $Z_t$ is bounded above by the process $\widetilde{Z}_t$ defined by 
\[d\widetilde{Z}_t=(r\widetilde{Z}_t+(\alpha-r) kY_t -c)dt+\sigma k Y_t dW_t, \quad \widetilde{Z}_{\tau}=\xi^0+k\xi^1 \ \text{ for } t<\tau_d,\]
and $\widetilde{Z}_t=\bfDelta$ for $t\geq \tau_d$. $\widetilde{Z}_t$ is the wealth process if the amount invested in the (frictionless) stock market is $kY_t$.
Let $f(x):=\psi_k(x,0)\in C^1[b,c/r]\cap C^2[b,c/r)$. We have $\psi_k(x,y)=f(x+ky)$. Since $f$ is decreasing in $[b,\infty)$, we deduce
\begin{equation}\label{eq:phi_k:1}
\mathbb{E}[\psi_k(X_\rho, Y_\rho)|\mathcal{G}_\tau]=\mathbb{E}[1_{\{\rho<\tau_d\}}f(Z_\rho)|\mathcal{G}_\tau]\geq \mathbb{E}[1_{\{\rho<\tau_d\}}f(\widetilde{Z}_\rho)|\mathcal{G}_\tau]=\mathbb{E}[f(\widetilde{Z}_\rho)|\mathcal{G}_\tau].
\end{equation}
In the event  $A:=\{\widetilde{Z}_{\tau}\in[c/r,\infty)\cup\{\bfDelta\}\}\in\mathcal{G}_\tau$, we have $f(\widetilde{Z}_\rho)\geq 0=f(\widetilde{Z}_\tau)$. In the event $A^c:=\{\widetilde{Z}_{\tau}\in [b,c/r)\}$, we let $\nu:=\inf\{t\geq 0: \widetilde{Z}_{\rho}\in [c/r,\infty)\}$, and use $f$ is non-negative in $[b,\infty)$ and zero in $[c/r,\infty)$ to get $f(\widetilde{Z}_\rho)\geq f(\widetilde{Z}_{\rho\wedge\nu})$. We therefore have
\begin{equation}\label{eq:phi_k:2}
\mathbb{E}[f(\widetilde{Z}_\rho)|\mathcal{G}_\tau]\geq \mathbb{E}[1_Af(\widetilde{Z}_\tau)+1_{A^c}f(\widetilde{Z}_{\rho\wedge\nu})|\mathcal{G}_\tau].
\end{equation}
In the event $A^c$, we use It\^{o}'s formula to obtain
\begin{align*}
f(\widetilde{Z}_{\rho\wedge\nu})&=f(\widetilde{Z}_\tau)+\int_\tau^{\rho\wedge\nu} \left\{f'(\widetilde{Z}_t)[r\widetilde{Z}_t+(\alpha-r) kY_t -c]+\frac{1}{2}f''(\widetilde{Z}_t)\sigma^2 (kY_t)^2-\beta f(\widetilde{Z}_t)\right\} dt\\
&\hspace{1.5cm} +\int_\tau^{\rho\wedge\nu} f'(\widetilde{Z}_t) \sigma kY_t dW_t +\int_\tau^{\rho\wedge \nu} [f(\bfDelta)-f(\widetilde{Z}_{t-})] d(N_t-\beta t).
\end{align*}
Notice that $f$ is the frictionless value function which satisfies the HJB equation
\[\beta f(x)=\inf_{\pi}\left\{\frac{1}{2}f''(x)\pi^2+(\alpha-r)f'(x)\pi+(rx-c)f'(x)\right\}\]
in $[b,c/r)$. It follows that the drift term is non-negative. For $t\in[\tau,\rho\wedge\nu]$, the process $\widetilde{Z}_t\in [b,c/r]$. So $Z_t\in [b,c/r]$ and the process $(X_t,Y_t)$ stays inside the bounded set $\{(x,y)\in\ubar{\mathcal{S}} :x+ky\leq c/r\}$. Here it is crucial that $k\in(1-\mu,\frac{1}{1-\lambda})$ for $Y_t$ to be bounded. The integrals with respect to the martingales $W_t$ and $N_t-\beta t$ then vanish upon taking $\mathcal{G}_\tau$-conditional expectation. This leads to
\begin{equation}\label{eq:phi_k:3}
\mathbb{E}[1_{A^c}f(\widetilde{Z}_{\rho\wedge\nu})|\mathcal{G}_\tau]\geq \mathbb{E}[1_{A^c}f(\widetilde{Z}_\tau)|\mathcal{G}_\tau].
\end{equation}
Putting \eqref{eq:phi_k:1}, \eqref{eq:phi_k:2} and \eqref{eq:phi_k:3} together, we get
\[\mathbb{E}[\psi_k(X_\rho, Y_\rho)|\mathcal{G}_\tau]\geq f(\widetilde{Z}_\tau)=1_{\{\tau<\tau_d\}}f(\xi^0+k\xi^1)=1_{\{\tau<\tau_d\}}\psi_k(\xi)=\psi_k(\xi)\]
which is the desired submartingale property.
\end{proof}

\begin{prop}\label{prop:v- is a vsupsoln}
The \textit{lower stochastic envelope}
\[v_-(x,y):=\sup_{v\in \mathcal{V}^-} v(x,y)\]
is a viscosity supersolution of \eqref{HJB} satisfying $v_-\geq 1$ on $\partial \mathcal{S}_b$ and $v_-\geq 0$ on $\partial \mathcal{S}_{c/r}$. 
\end{prop}
\begin{proof}
The boundary inequalities are satisfied because $v_-\geq \lbar{\psi}$ by Lemma \ref{lemma:lbarpsi is a ssubsoln}.\footnote{In fact, equalities hold for $v_-$ on the boundary; the reverse inequalities holds because (SB1) is preserved under pointwise maximum.} 
To show interior viscosity supersolution property, let $(x_0, y_0)\in \mathcal{S}$ and $\varphi\in C^2(\mathcal{S})$ be a test function such that $v_--\varphi$ attains a strict minimum of zero at $(x_0, y_0)$. We need to show
\[\max\left\{\mathcal{L}\varphi, -(1-\mu)\varphi_x+\varphi_y, \varphi_x-(1-\lambda)\varphi_y\right\}(x_0, y_0)\geq 0.\]
Assume on the contrary that 
\[\max\left\{\mathcal{L}\varphi, -(1-\mu)\varphi_x+\varphi_y, \varphi_x-(1-\lambda)\varphi_y\right\}(x_0, y_0)< 0.\]
Similar to the proof of Proposition \ref{prop:v+ is a vsubsoln}, we can find $0<\epsilon<1$, $\eta>0$ and $v\in \mathcal{V}^-$ such that $\varphi^\eta:=\varphi+\eta$ satisfies
\begin{equation}\label{vsupsoln_eq1}
\max\left\{\mathcal{L}\varphi^\eta, -(1-\mu)\varphi^\eta_x+\varphi^\eta_y, \varphi^\eta_x-(1-\lambda)\varphi^\eta_y\right\}< 0 \quad \text{on}\quad \ubar{B_\epsilon(x_0,y_0)},
\end{equation}
\begin{equation}\label{vsupsoln_eq2}
\varphi^\eta \leq v\quad \text{on} \quad \ubar{B_\epsilon(x_0, y_0)}\backslash B_{\epsilon/2}(x_0, y_0),
\end{equation}
and
\begin{equation*}
 \varphi^\eta(x_0, y_0)>v_-(x_0, y_0).
\end{equation*}
The technique for constructing the lifting function $\varphi^\eta$ is classical and similar to the stochastic supersolution case. So we skip the details. Define
\[v^\eta:=\begin{cases}
v\vee \varphi^\eta &\text{ on } \ubar{B_\epsilon(x_0,y_0)},\\
v &\text{ on } \ubar{B_\epsilon(x_0,y_0)}^c.
\end{cases}\]
It suffices to show $v^\eta\in\mathcal{V}^-$. And the only nontrivial part is to check $v^\eta$ satisfies (SB2). 

Let $(\tau,\xi)$ be any random initial condition, $(B,S)$ be any $(\tau,\xi)$-admissible control and $\rho\in[\tau, \sigma^{\tau,\xi,B,S}]$ be any $\mathbb{G}$-stopping time. Let
\[A:=\{\xi\in B_{\epsilon/2}(x_0,y_0)\}\cap\{\varphi^\eta(\xi)>v(\xi)\}\in\mathcal{G}_\tau.\]
Let
\[\tau_1:=\inf\{t\in[\tau,\sigma^{\tau,\xi,B,S}]: (X^{\tau,\xi,B,S}_t, Y_t^{\tau,\xi,B,S})\notin B_{\epsilon/2}(x_0, y_0)\}\]
and
\[\xi_1:=(X^{\tau,\xi,B,S}_{\tau_1}, Y_{\tau_1}^{\tau,\xi,B,S})\in\mathcal{G}_{\tau_1}.\]
In the event $A$, because of a possible jump transaction at time $\tau_1$, $\xi_1$ may not be on $\partial B_{\epsilon/2}(x_0,y_0)\cup\{\bfDelta\}$. This will bring some problem since \eqref{vsupsoln_eq1} is only valid locally. To overcome this issue, we define an intermediate position $\xi'_1$ as follows: let $\xi_{1-}:=(X^{\tau,\xi,B,S}_{\tau_1-}, Y_{\tau_1-}^{\tau,\xi,B,S})$. We have $\xi_{1-}\in \ubar{B_{\epsilon/2}(x_0,y_0)}$ on $A$. Define 
\[\xi'_1:=1_{A\cap\{\tau_1<\tau_d\}}\left(1_{\{\triangle B_{\tau_1}>0\}}\mathfrak{b}(\xi_{1-})+1_{\{\triangle S_{\tau_1}> 0\}}\mathfrak{s}(\xi_{1-})\right)+1_{A^c\cup\{\tau_1=\tau_d\}}\xi_1\in\mathcal{G}_{\tau_1},\]
where $\mathfrak{b}, \mathfrak{s}$ are the functions introduced in cases (ii) and (iii) of the proof of Proposition \ref{prop:v+ is a vsubsoln}. On $A\cap\{\tau_1<\tau_d\}$, $\xi'_1$ is the intersection of $\partial B_{\epsilon/2}(x_0, y_0)$ and the line segment connecting $\xi_{1-}$ and $\xi_1$. 
Also define $(B^1,S^1)$ by
\begin{align*}
(\triangle B^1_{\tau_1}, \triangle S^1_{\tau_1}):=&1_{A\cap\{\tau_1<\tau_d\}}\left(1_{\{\triangle B_{\tau_1}> 0\}}(\xi^0_{1-}-\mathfrak{b}^0(\xi_{1-}),0)+1_{\{\triangle S_{\tau_1}> 0\}}(0,\xi^1_{1-}-\mathfrak{s}^1(\xi_{1-}))\right)\\
&\quad +1_{A^c\cup\{\tau_1=\tau_d\}}(\triangle B_{\tau_1},\triangle S_{\tau_1})
\end{align*}
and
\[(B^1_t,S^1_t):=1_{\{t<\tau_1\}}(B_t, S_t)+1_{\{t\geq \tau_1\}}\left[(B_{\tau_1-}, S_{\tau_1-})+(\triangle B^1_{\tau_1}, \triangle S^1_{\tau_1})\right].\]
That is, $(B^1,S^1)$ agrees with $(B,S)$ before time $\tau_1$, but at time $\tau_1$, the corresponding controlled process only jumps to $\xi'_1$ instead of $\xi_1$. We have $(B^1,S^1)\in\mathscr{A}_0$ and $(X^{\tau,\xi,B^1,S^1}_t, Y^{\tau,\xi,B^1,S^1}_t)\in \ubar{B_{\epsilon/2}(x_0,y_0)}\cup\{\bfDelta\}$ for all $t\in[\tau,\tau_1]$ on $A$.
Apply generalized It\^{o}'s formula to the RCLL semimartingale $\varphi^\eta(X^{\tau,\xi,B^1,S^1},Y^{\tau,\xi,B^1,S^1})$ on $A$,
we get
\begin{align*}
\varphi^\eta(X^{\tau,\xi,B^1,S^1}_{\rho\wedge\tau_1},Y^{\tau,\xi,B^1,S^1}_{\rho\wedge\tau_1})&=\varphi^\eta(X^{\tau,\xi,B^1,S^1}_{\tau},Y^{\tau,\xi,B^1,S^1}_{\tau})+\int_\tau^{\rho\wedge\tau_1} -\mathcal{L}\varphi^\eta(X^{\tau,\xi,B^1,S^1}_t,Y^{\tau,\xi,B^1,S^1}_t)dt\\
&\quad +\int_\tau^{\rho\wedge \tau_1}(\varphi^\eta)'(X^{\tau,\xi,B^1,S^1}_t, Y^{\tau,\xi,B^1,S^1}_t)\sigma Y^{\tau,\xi,B^1,S^1}_t dW_t\\
&\quad +\int_\tau^{\rho\wedge \tau_1}[-\varphi^\eta_x+(1-\lambda)\varphi^\eta_y](X^{\tau,\xi,B^1,S^1}_t, Y^{\tau,\xi,B^1,S^1}_t)dB^c_t\\
&\quad +\int_\tau^{\rho\wedge \tau_1}[(1-\mu)\varphi^\eta_x-\varphi^\eta_y](X^{\tau,\xi,B^1,S^1}_t, Y^{\tau,\xi,B^1,S^1}_t)dS^c_t\\
&\quad +\int_\tau^{\rho\wedge \tau_1}[\varphi^\eta(\bfDelta)-\varphi^\eta(X^{\tau,\xi,B^1,S^1}_{t-}, Y^{\tau,\xi,B^1,S^1}_{t-})]d(N_t-\beta t)\\
&\quad+\sum_{\substack{\tau\leq t\leq \rho\wedge \tau_1\\ t<\tau_d}} \varphi^\eta(X^{\tau,\xi,B^1,S^1}_{t}, Y^{\tau,\xi,B^1,S^1}_{t})-\varphi^\eta(X^{\tau,\xi,B^1,S^1}_{t-}, Y^{\tau,\xi,B^1,S^1}_{t-})
\end{align*}
where $B^c, S^c$ denote the continuous part of $B, S$. By \eqref{vsupsoln_eq1}, the $dt$, $dB^c$ and $dS^c$ integrals are non-negative. The $dW$ integral and the integral with respect to the compensated Poisson process vanish if we take $\mathcal{G}_\tau$-conditional expectation. We now analyze the last term which represents contribution from jump transactions. Similar to case (ii) of the proof of Proposition \ref{prop:v+ is a vsubsoln} (see \eqref{ssupsoln_ii_eq2}), we can use \eqref{vsupsoln_eq1} and Mean Value Theorem to deduce
\[\varphi^\eta(x-h,y+(1-\lambda)h)\geq \varphi^\eta(x,y),\]
and
\[\varphi^\eta(x+(1-\mu)h',y-h')\geq \varphi^\eta(x,y).\]
for all $(x,y)\in B_{\epsilon}(x_0,y_0)$ and $h,h'>0$ such that $(x-h,y+(1-\lambda)h), (x+(1-\mu)h',y-h')\in B_{\epsilon}(x_0,y_0)$. It follows that on the set $A$ and for $t\in[\tau, \tau_1]\backslash\{\tau_d\}$, if $\triangle B^1_t>0$, then
\begin{align*}
\varphi^\eta(X^{\tau,\xi,B^1,S^1}_{t}, Y^{\tau,\xi,B^1,S^1}_{t})&= \varphi^\eta(X^{\tau,\xi,B^1,S^1}_{t-}-\triangle B^1_t, Y^{\tau,\xi,B^1,S^1}_{t-}+(1-\lambda)\triangle B^1_t)\\
&\geq \varphi^\eta(X^{\tau,\xi,B^1,S^1}_{t-}, Y^{\tau,\xi,B^1,S^1}_{t-}).
\end{align*}
If $\triangle S^1_t>0$, then
\begin{align*}
\varphi^\eta(X^{\tau,\xi,B^1,S^1}_{t}, Y^{\tau,\xi,B^1,S^1}_{t})&=\varphi^\eta(X^{\tau,\xi,B^1,S^1}_{t-}+(1-\mu)\triangle S^1_t, Y^{\tau,\xi,B^1,S^1}_{t-}-\triangle S^1_t)\\
&\geq \varphi^\eta(X^{\tau,\xi,B^1,S^1}_{t-}, Y^{\tau,\xi,B^1,S^1}_{t-}).
\end{align*}
Since $\triangle B^1_t$ and $\triangle S^1_t$ are not positive at the same time (see the definition of $\mathscr{A}_0$), each summand in the last term is non-negative. Putting everything together, we obtain by taking $\mathcal{G}_\tau$-conditional expectation of the expression given by It\^{o}'s formula that
\begin{equation*}
\mathbb{E}[1_A\varphi^\eta(X^{\tau,\xi,B^1,S^1}_{\rho\wedge\tau_1},Y^{\tau,\xi,B^1,S^1}_{\rho\wedge\tau_1})|\mathcal{G}_\tau]\geq 1_A\varphi^\eta(X^{\tau,\xi,B^1,S^1}_{\tau},Y^{\tau,\xi,B^1,S^1}_{\tau}).
\end{equation*}
Again, we use that $\varphi^\eta$ is non-decreasing if we move northwest along the vector $(-1,1-\lambda)$ and southeast along the vector $(1-\mu,-1)$ inside the ball $B_\epsilon(x_0,y_0)$ to bound the right hand side from below by
\[1_{A\cap\{\tau<\tau_d\}}\varphi^\eta (\xi)+1_{A\cap\{\tau=\tau_d\}}\varphi^\eta(\bfDelta)=1_A\varphi^\eta(\xi)=1_A v^\eta(\xi).\]
For the left hand side, we use $v^\eta\geq \varphi^\eta$ in $B_\epsilon(x_0,y_0)$ and that $(B^1,S^1)=(B,S)$ before $\tau_1$ to obtain
\begin{align*}
1_A\varphi^\eta(X^{\tau,\xi,B^1,S^1}_{\rho\wedge\tau_1},Y^{\tau,\xi,B^1,S^1}_{\rho\wedge\tau_1})&\leq 1_{A}v^\eta(X^{\tau,\xi,B^1,S^1}_{\rho\wedge\tau_1},Y^{\tau,\xi,B^1,S^1}_{\rho\wedge\tau_1})\\
&=1_{A\cap\{\rho<\tau_1\}}v^\eta(X^{\tau,\xi,B,S}_{\rho},Y^{\tau,\xi,B,S}_{\rho})+1_{A\cap\{\rho\geq\tau_1\}}v^\eta(\xi'_1).
\end{align*}
Hence
\begin{equation}\label{vsupsoln_eq3}
\mathbb{E}[1_{A\cap\{\rho<\tau_1\}}v^\eta(X^{\tau,\xi,B,S}_{\rho},Y^{\tau,\xi,B,S}_{\rho})+1_{A\cap\{\rho\geq\tau_1\}}v^\eta(\xi'_1)|\mathcal{G}_\tau]\geq 1_A v^\eta(\xi).
\end{equation}
Define
\[(B^2_t,S^2_t):=1_{\{t\geq \tau_1\}}[(B_{\tau_1},S_{\tau_1})-(B^1_{\tau_1},S^1_{\tau_1})].\]
Starting with the random initial condition $(\tau_1,\xi'_1)$, $(B^2,S^2)$ immediately brings the state process from $\xi'_1$ back to $\xi_1$ and stays inactive afterwards. It is easy to see that $(X^{\tau_1,\xi'_1,B^2,S^2},Y^{\tau_1,\xi'_1,B^2,S^2})$ either exit $\mathcal{S}$ at time $\tau_1$ with exit position $\xi_1$, or at a later time when the control is inactive so that the exit is caused by diffusion or death. In both cases, the exit position belongs to $\partial\mathcal{S}\cup\{\bfDelta\}$. So $(B^2, S^2)\in \mathscr{A}(\tau_1,\xi'_1)$. Using the submartingale property of $v(X^{\tau_1,\xi'_1,B^2,S^2}, Y^{\tau_1,\xi'_1,B^2,S^2})$, we have
\[v^\eta(\xi_1)=v(\xi_1)=\mathbb{E}[v(X^{\tau_1,\xi'_1,B^2,S^2}_{\tau_1},Y^{\tau_1,\xi'_1,B^2,S^2}_{\tau_1})|\mathcal{G}_{\tau_1}]\geq v(\xi'_1)=v^\eta(\xi'_1),\]
where the first and the last equalities hold because $\xi_1,\xi'_1\notin B_{\epsilon/2}(x_0,y_0)$. \eqref{vsupsoln_eq3} then implies
\begin{equation}\label{vsupsoln_eq3b}
\mathbb{E}[1_{A\cap\{\rho<\tau_1\}}v^\eta(X^{\tau,\xi,B,S}_{\rho},Y^{\tau,\xi,B,S}_{\rho})+1_{A\cap\{\rho\geq\tau_1\}}v^\eta(\xi_1)|\mathcal{G}_\tau]\geq 1_A v^\eta(\xi).
\end{equation}

On the set $A^c$, we use the submartingale property (SB2) of $v(X^{\tau,\xi,B,S}, Y^{\tau,\xi,B,S})$ to get
\begin{equation*}
\mathbb{E}[1_{A^c}v^\eta(X^{\tau,\xi,B,S}_{\rho\wedge\tau_1},Y^{\tau,\xi,B,S}_{\rho\wedge\tau_1})|\mathcal{G}_\tau]\geq\mathbb{E}[1_{A^c}v(X^{\tau,\xi,B,S}_{\rho\wedge\tau_1},Y^{\tau,\xi,B,S}_{\rho\wedge\tau_1})|\mathcal{G}_\tau]\geq 1_{A^c}v(\xi)=1_{A^c}v^\eta(\xi),
\end{equation*}
or
\begin{equation}\label{vsupsoln_eq4}
\mathbb{E}[1_{A^c\cap\{\rho<\tau_1\}}v^\eta(X^{\tau,\xi,B,S}_{\rho},Y^{\tau,\xi,B,S}_{\rho})+1_{A^c\cap\{\rho\geq\tau_1\}}v^\eta(\xi_1)|\mathcal{G}_\tau]\geq1_{A^c}v^\eta(\xi).
\end{equation}
Adding \eqref{vsupsoln_eq3b} and \eqref{vsupsoln_eq4} yields
\begin{equation}\label{vsupsoln_eq5}
\mathbb{E}[1_{\{\rho<\tau_1\}}v^\eta(X^{\tau,\xi,B,S}_{\rho},Y^{\tau,\xi,B,S}_{\rho})+1_{\{\rho\geq \tau_1\}}v^\eta(\xi_1)|\mathcal{G}_\tau]\geq v^\eta(\xi).
\end{equation}
Let 
\[(B^3_t,S^3_t):=(B_t,S_t)-1_{\{t\geq \tau_1\}}(\triangle B_{\tau_1}, \triangle S_{\tau_1})\]
be the same control as $(B,S)$, but with any jump transaction at time $\tau_1$ removed. We have 
\begin{equation}\label{vsupsoln_eq5.5}
(X^{\tau,\xi,B,S}_t, Y^{\tau,\xi,B,S}_t)=(X^{\tau_1,\xi_1, B^3, S^3}_t, Y^{\tau_1,\xi_1, B^3, S^3}_t) \quad \forall\, t\geq \tau_1.
\end{equation}
The reason for introducing another control is because our random initial condition allows a jump at initial time. Since $\xi_1$ already includes the possible jump transactions specified by $(B,S)$ at time $\tau_1$, we want to avoid doing the same transaction again when using $(\tau_1, \xi_1)$ as the new random initial condition. That is, $(B^3, S^3)$ is defined to make \eqref{vsupsoln_eq5.5} hold. To see $(B^3, S^3)\in\mathscr{A}(\tau_1,\xi_1)$, first notice that $\sigma^{\tau,\xi,B,S}\geq \tau_1$ by the definition of $\tau_1$. \eqref{vsupsoln_eq5.5} then implies $\sigma^{\tau_1,\xi_1,B^3,S^3}=\sigma^{\tau,\xi,B,S}$. Thus,
\[(X^{\tau_1,\xi_1, B^3, S^3}_{\sigma^{\tau_1,\xi_1,B^3,S^3}}, Y^{\tau_1,\xi_1, B^3, S^3}_{\sigma^{\tau_1,\xi_1,B^3,S^3}})=(X^{\tau,\xi, B, S}_{\sigma^{\tau,\xi,B,S}}, Y^{\tau,\xi, B, S}_{\sigma^{\tau,\xi,B,S}})\in\partial\mathcal{S}\cup\{\bfDelta\}\]
by the $(\tau,\xi)$-admissibility of $(B,S)$. The submartingale property (SB2) of $v(X^{\tau_1,\xi_{1},B^3,S^3}, Y^{\tau_1,\xi_{1},B^3,S^3})$ (applied to the stopping time $\rho\vee\tau_1$) implies
\begin{equation*}
\begin{aligned}
\mathbb{E}[1_{\{\rho\geq \tau_1\}}v^\eta(X^{\tau,\xi,B,S}_\rho, Y^{\tau,\xi,B,S}_\rho)|\mathcal{G}_{\tau_1}]&=\mathbb{E}[1_{\{\rho\geq \tau_1\}}v^\eta(X^{\tau_1,\xi_1,B^3,S^3}_{\rho},Y^{\tau_1,\xi_1,B^3,S^3}_{\rho})|\mathcal{G}_{\tau_1}]\\
&\geq\mathbb{E}[1_{\{\rho\geq \tau_1\}}v(X^{\tau_1,\xi_1,B^3,S^3}_{\rho},Y^{\tau_1,\xi_1,B^3,S^3}_{\rho})|\mathcal{G}_{\tau_1}]\\
&\geq 1_{\{\rho\geq \tau_1\}}v(\xi_1)=1_{\{\rho\geq \tau_1\}}v^\eta(\xi_1)
\end{aligned}
\end{equation*}
Taking $\mathcal{G}_\tau$-conditional expectation, we get
\begin{equation}\label{vsupsoln_eq6}
\mathbb{E}[1_{\{\rho\geq \tau_1\}}v^\eta(X^{\tau,\xi,B,S}_\rho, Y^{\tau,\xi,B,S}_\rho)|\mathcal{G}_{\tau}]\geq \mathbb{E}[1_{\{\rho\geq \tau_1\}}v^\eta(\xi_1)|\mathcal{G}_\tau].
\end{equation}
Adding \eqref{vsupsoln_eq5} and \eqref{vsupsoln_eq6}, we get
\begin{equation*}
\mathbb{E}[v^\eta(X^{\tau,\xi,B,S}_\rho, Y^{\tau,\xi,B,S}_\rho)|\mathcal{G}_\tau]\geq v^\eta(\xi).
\end{equation*}
This completes the verification of (SB2) for $v^\eta$, and hence of the viscosity supersolution property of $v_-$.
\end{proof}

\section{Comparison Principle}\label{sec:comparison}

A comparison principle can be established following the idea of \cite{Kabanov04}. The key is to show the existence of a strict subsolution which is then added to the penalty term when applying the technique of doubling of variables. We give a proof here for the sake of completeness.

\begin{lemma}\label{Lyapunov_function}
There exists a strict subsolution $\ell$ of \eqref{HJB} satisfying
\begin{itemize}
\item[(1)] $\ell\in C^2(\ubar{\mathcal{S}})$ and $\ell<0$;
\item[(2)] $\ell(x,y)\rightarrow -\infty$ as $\|(x,y)\|\rightarrow \infty$ in $\ubar{\mathcal{S}}$.\footnote{The function $\ell$ is referred to as a Lyapunov function in \cite{Kabanov04}.
}
\end{itemize}
\end{lemma}
\begin{proof}
Let $h(z):=-\frac{(z-b+1)^p}{p}$ with $0<p<1$. We have $h<0$, $h'<0$ and $h''>0$ in $(b-1,\infty)$. Let $1-\mu<k<\frac{1}{1-\lambda}$ and define $\ell(x,y):=h(x+ky)$. $\ell$ is well-defined since $x+ky\geq b$ for all $(x,y)\in\ubar{\mathcal{S}}$. Condition (1) is trivially satisfied. To see condition (2) holds, observe that for each $a>b$, $\{(x,y)\in\ubar{\mathcal{S}}: x+ky\leq a\}$ is a bounded subset of $\mathbb{R}^2$. Therefore if $\|(x,y)\|\rightarrow \infty$ in $\ubar{\mathcal{S}}$, then we must have $x+ky\rightarrow \infty$. It follows that $\ell(x,y)=h(x+ky)\rightarrow -\infty$. It remains to show $\ell$ is a strict subsolution of \eqref{HJB} under a suitable choice of $p$.

Let $(x,y)\in\mathcal{S}$. By our choice of $k$ and that $h'<0$, we readily obtain 
\[-(1-\mu)\ell_x+\ell_y=[-(1-\mu)+k]h'(x+ky)<0\]
and
\[\ell_x-(1-\lambda)\ell_y=[1-k(1-\lambda)]h'(x+ky)<0.\]
Let us now compute $\mathcal{L}\ell(x,y)$.
\begin{align*}
\mathcal{L}\ell(x,y)&=\beta \ell(x,y)-(rx-c)\ell_x(x,y)-\alpha y \ell_y(x,y)-\frac{1}{2}\sigma^2 y^2\ell_{yy}(x,y)\\
&=\beta h(x+ky)-(rx-c+\alpha k y)h'(x+ky)-\frac{1}{2}\sigma^2 y^2 k^2h''(x+ky).
\end{align*}
By definition of the solvency region $\mathcal{S}$, we have
\[x+(1-\mu)y<\frac{c}{r} \ \text{ if } \ y>0, \text{ and } \ x+\frac{y}{1-\lambda}<\frac{c}{r}\ \text{ if } \ y<0,\]
which implies
\[rx-c+\alpha k y\leq \frac{r|y|}{1-\lambda}+\alpha k |y|=\left(\frac{r}{1-\lambda}+\alpha k\right)|y|:=\theta |y|.\]
Using $h'(x+ky)<0$ and $h''(x+ky)>0$, we deduce
\begin{align*}
\mathcal{L}\ell(x,y)&\leq \beta h(x+ky)-\theta |y|h'(x+ky)-\frac{1}{2}\sigma^2 y^2 k^2h''(x+ky)\\
&= -\frac{1}{2}\left[\sigma^2 y^2 k^2h''(x+ky)+2\theta|y|h'(x+ky)+\frac{\theta^2 (h'(x+ky))^2}{\sigma^2k^2 h''(x+ky)}\right]\\
&\quad +\beta h(x+ky)+\frac{1}{2}\frac{\theta^2 (h'(x+ky))^2}{\sigma^2k^2 h''(x+ky)}\\
&\leq \left(\beta +\frac{1}{2}\frac{\theta^2}{\sigma^2k^2}\frac{(h')^2}{hh''}(x+ky)\right)h(x+ky)\\
&=\left(\beta -\frac{1}{2}\frac{\theta^2}{\sigma^2k^2}\frac{p}{1-p}\right)h(x+ky).
\end{align*}
Choose $p$ small such that $\beta>\frac{1}{2}\frac{\theta^2}{\sigma^2k^2}\frac{p}{1-p}$. We then have by negativity of $h$ that $\mathcal{L}\ell(x,y)<0$.
\end{proof}

\begin{prop}\label{prop:comparison}
Let $u, v$ be u.s.c. viscosity subsolution and l.s.c. viscosity supersolution of \eqref{HJB}, respectively. Suppose $u, v$ are bounded and $u\leq v$ on $\partial\mathcal{S}$, then $u\leq v$ in $\mathcal{S}$.
\end{prop}
\begin{proof}
Assume to the contrary that $\delta:=u(x_0,y_0)-v(x_0,y_0)>0$ for some $(x_0, y_0)\in\mathcal{S}$. Let $\ell$ be the strict classical subsolution given by Lemma \ref{Lyapunov_function}. Let $\epsilon$ be a small positive constant satisfying $\delta+2\epsilon \ell(x_0,y_0)>0$. For each $\theta>0$, define
\begin{align*}
\Phi_\theta(x,y,x',y'):=&u(x,y)-v(x',y')-\frac{\theta}{2}(|x-x'|^2+|y-y'|^2)+\epsilon \ell(x,y)+\epsilon \ell(x',y').
\end{align*}
Since $u(x,y)-v(x',y')$ is u.s.c. and bounded, and $\ell(x,y)\rightarrow -\infty$ as $\|(x,y)\|\rightarrow \infty$ in $\ubar{\mathcal{S}}$, there exists $(x_\theta, y_\theta), (x'_\theta, y'_\theta)$ lying in a compact subset of $\ubar{\mathcal{S}}$ such that
\[\sup_{(x,y),(x',y')\in\ubar{\mathcal{S}}}\Phi_\theta(x,y,x',y')=\Phi_\theta(x_\theta,y_\theta,x'_\theta,y'_\theta).\]
Compactness allows us to extract a sequence $\theta_n\rightarrow \infty$ such that $(x_n, y_n, x'_n, y'_n):=(x_{\theta_n}, y_{\theta_n}, x'_{\theta_n}, y'_{\theta_n})$ $\rightarrow (\hat{x}, \hat{y}, \hat{x}', \hat{y}')$ as $n\rightarrow \infty$. Clearly, we have
\begin{equation}\label{comparison_eq0}
\Phi_{\theta_n}(x_n, y_n,x'_n, y'_n)\geq \sup_{(x,y)\in \ubar{\mathcal{S}}} \Phi_0(x,y,x,y)\geq \delta+2\epsilon \ell(x_0,y_0)>0.
\end{equation}
It follows that
\begin{align*}
\frac{\theta_n}{2}(|x_n-x'_n|^2+|y_n-y'_n|^2)\leq & u(x_n,y_n)-v(x'_n,y'_n) +\epsilon \ell(x_n,y_n)+\epsilon \ell(x'_n,y'_n)-\sup_{(x,y)\in \ubar{\mathcal{S}}} \Phi_0(x,y,x,y).
\end{align*}
Since the right hand side is bounded from above and $\theta_n\rightarrow \infty$, we must have $|x_n-x'_n|^2+|y_n-y'_n|^2\rightarrow 0$, hence $(\hat{x}, \hat{y})=(\hat{x}', \hat{y}')$. This further implies by u.s.c. of $u-v$ that
\[0\leq \limsup_n \frac{\theta_n}{2}(|x_n-x'_n|^2+|y_n-y'_n|^2)\leq \Phi_0(\hat{x},\hat{y},\hat{x},\hat{y})-\sup_{(x,y)\in \ubar{\mathcal{S}}} \Phi_0(x,y,x,y)\leq 0.\]
So we conclude
\begin{equation}\label{comparison_eq1}
\lim_n \theta_n(|x_n-x'_n|^2+|y_n-y'_n|^2)=0,
\end{equation}
and
\begin{equation}\label{comparison_eq2}
\lim_n \Phi_{\theta_n}(x_n, y_n,x'_n, y'_n)=\Phi_0(\hat{x},\hat{y},\hat{x},\hat{y})=\sup_{(x,y)\in \ubar{\mathcal{S}}} \Phi_0(x,y,x,y)>0.
\end{equation}

Now, since $u\leq v$ on $\partial \mathcal{S}$ and $\ell\leq 0$, we have $\Phi_0(x,y,x,y)\leq 0$ for $(x,y)\in\partial \mathcal{S}$. In view of \eqref{comparison_eq2}, we have $(\hat{x}, \hat{y})\in \mathcal{S}$. So $(x_n, y_n), (x'_n, y'_n)\in\mathcal{S}$ for $n$ sufficiently large. By Crandall-Ishii's lemma, we can find matrices $A_n, B_n\in\mathbb{S}_2$ such that
\begin{equation}\label{comparison_eq3}
\left(\theta_n(x_n-x'_n), \theta_n(y_n-y'_n), A_n\right)\in \bar{J}^{2,+}_{\mathcal{S}}\big(u(x_n, y_n)+\epsilon \ell(x_n,y_n)\big),
\end{equation}
\begin{equation}\label{comparison_eq4}
\left(\theta_n(x_n-x'_n), \theta_n(y_n-y'_n), B_n\right)\in \bar{J}^{2,-}_{\mathcal{S}}\big(v(x'_n, y'_n)-\epsilon \ell(x'_n,y'_n)\big),
\end{equation}
and
\begin{equation*}
\begin{pmatrix}
A_n & 0\\ 0 & -B_n
\end{pmatrix}\leq 3\theta_n
\begin{pmatrix}
I & -I \\ -I & I
\end{pmatrix}.
\end{equation*}
where $\bar{J}^{2,+}_{\mathcal{S}}$ and $\bar{J}^{2,-}_{\mathcal{S}}$ denote the closure of the second order superjet and subjet, respectively. By Lemma 4.2.7 of \cite{Kabanov}, we have
\begin{equation}\label{comparison_eq5}
(y_n)^2A_{n,22}-(y'_n)^2B_{n,22}\leq 3\theta_n |y_n-y'_n|^2.
\end{equation}
Since $\ell$ is a $C^2(\mathcal{S})$ functions, we can rewrite \eqref{comparison_eq3} and \eqref{comparison_eq4} as
\begin{equation*}
(p_n, X_n)\in \bar{J}^{2,+}_{\mathcal{S}}u(x_n, y_n), \quad (q_n, Y_n)\in \bar{J}^{2,-}_{\mathcal{S}}v(x'_n, y'_n)
\end{equation*}
where
\begin{align*}
p_n&:=\theta_n(x_n-x'_n, y_n-y'_n)-\epsilon D\ell(x_n, y_n), \ X_n:=A_n-\epsilon D^2 \ell (x_n, y_n),\\
q_n&:=\theta_n(x_n-x'_n, y_n-y'_n)+\epsilon D\ell(x'_n, y'_n),\ Y_n:=B_n+\epsilon D^2 \ell(x'_n, y'_n).
\end{align*} 
By the semijets definition of viscosity solution, we have
\[\max\bigg\{\beta u(x_n,y_n)-(rx_n-c)p_{n,1}-\alpha y_n p_{n,2}-\frac{1}{2}\sigma^2y_n^2 X_{n,22}, -(1-\mu)p_{n,1}+p_{n,2}, p_{n,1}-(1-\lambda)p_{n,2}\bigg\}\leq 0\]
and
\[\max\bigg\{\beta v(x'_n,y'_n)-(rx'_n-c)q_{n,1}-\alpha y'_n q_{n,2}-\frac{1}{2}\sigma^2(y'_n)^2 Y_{n,22}, -(1-\mu)q_{n,1}+q_{n,2}, q_{n,1}-(1-\lambda)q_{n,2}\bigg\}\geq 0.\]
We consider three cases.

Case 1. $-(1-\mu)q_{n,1}+q_{n,2}\geq 0$ for infinitely many $n$'s. In this case,
\begin{align*}
0&\geq -(1-\mu)p_{n,1}+p_{n,2}-[-(1-\mu)q_{n,1}+q_{n,2}]\\
&= -\epsilon[-(1-\mu)\ell_x(x_n,y_n)+\ell_y(x_n,y_n)]-\epsilon[-(1-\mu)\ell_x(x'_n,y'_n)+\ell_y(x'_n,y'_n)].
\end{align*}
Letting $n\rightarrow \infty$ yields
\begin{align*}
0&\geq -2\epsilon [-(1-\mu)\ell_x(\hat{x},\hat{y})+\ell_y(\hat{x},\hat{y})],
\end{align*}
or
\[-(1-\mu)\ell_x(\hat{x},\hat{y})+\ell_y(\hat{x},\hat{y})\geq 0.\]
This is a contradiction to the strict subsolution property of $\ell$ in the sell region.

Case 2. $q_{n,1}-(1-\lambda)q_{n,2}\geq 0$ for infinitely many $n$'s. Similar to case 1, this leads to $\ell_x(\hat{x},\hat{y})-(1-\lambda)\ell_y(\hat{x},\hat{y})\geq 0$, contradicting the strict subsolution property of $\ell$ in the buy region.

Case 3. For $n$ sufficiently large, $\beta v(x'_n,y'_n)-(rx'_n-c)q_{n,1}-\alpha y'_n q_{n,2}-\frac{1}{2}\sigma^2(y'_n)^2 Y_{n,22}\geq 0$. In this case,
\begin{align*}
0&\leq \beta v(x'_n,y'_n)-(rx'_n-c)q_{n,1}-\alpha y'_n q_{n,2}-\frac{1}{2}\sigma^2(y'_n)^2 Y_{n,22}\\
&\qquad -\left[\beta u(x_n,y_n)-(rx_n-c)p_{n,1}-\alpha y_n p_{n,2}-\frac{1}{2}\sigma^2y_n^2 X_{n,22}\right]\\
&=-\beta\left[u(x_n,y_n)-v(x'_n,y'_n)\right]+\epsilon(\mathcal{L}\ell-\beta \ell)(x_n,y_n)+\epsilon(\mathcal{L}\ell-\beta \ell)(x'_n,y'_n)\\
&\qquad +r\theta_n(x_n-x'_n)^2+\alpha\theta_n(y_n-y'_n)^2+\frac{1}{2}\sigma^2\left[y^2_nA_{n,22}-(y'_n)^2B_{n,22}\right]\\
&\leq -\beta\left[u(x_n,y_n)-v(x'_n,y'_n)+\epsilon \ell(x_n,y_n)+\epsilon \ell(x'_n,y'_n) \right]\\
&\qquad +\left(r+\alpha+\frac{3}{2}\sigma^2\right)\theta_n(|x_n-x'_n|^2+|y_n-y'_n|^2)\\
&= -\beta \Phi_{\theta_n}(x_n,y_n,x'_n,y'_n)+\left(r+\alpha+\frac{3}{2}\sigma^2-\frac{\beta}{2}\right)\theta_n(|x_n-x'_n|^2+|y_n-y'_n|^2)\\
&\leq -\beta (\delta-2\epsilon\ell(x_0,y_0))+\left(r+\alpha+\frac{3}{2}\sigma^2-\frac{\beta}{2}\right)\theta_n(|x_n-x'_n|^2+|y_n-y'_n|^2).
\end{align*}
In the third step, we used the subsolution property of $\ell$ and \eqref{comparison_eq5}. In the fourth step, we used the definition of $\Phi_\theta$. In the last step, we used \eqref{comparison_eq0}.
Letting $n\rightarrow \infty$ and using \eqref{comparison_eq1}, we arrive at the contradiction $0\leq -\beta (\delta-2\epsilon\ell(x_0,y_0))<0$. The proof is complete.
\end{proof}

\noindent{\bf Proof of Theorem \ref{thm:main result}.}
By Remarks \ref{rmk:ssupsoln>=psi} and \ref{rmk:ssubsoln<=psi}, we have $v_-\leq\psi\leq v_+$. By Propositions \ref{prop:v+ is a vsubsoln} and \ref{prop:v- is a vsupsoln}, we know $v_+$ is a viscosity subsolution and $v_-$ is a viscosity supersolution of \eqref{HJB}. Moreover, $v_+\leq v_-$ on $\partial \mathcal{S}$. It is also clear that $v_+$ is u.s.c. and $v_-$ is l.s.c.. Comparison principle (Proposition \ref{prop:comparison}) then implies $v_+\leq v_-$. Therefore, $v_+=v_-=\psi$ is a continuous viscosity solution to the Dirichlet problem \eqref{HJB}, \eqref{BC}. Uniqueness also follows from the comparison principle.
\hfill{\qed}

\bibliographystyle{plain}
\bibliography{transaction_cost}{}

\end{document}